\documentclass[12pt]{amsart}
\usepackage{times,amsfonts,amsmath,amstext,amsbsy,amssymb,amsopn,amsthm,upref,eucal}
\usepackage{amssymb}
\usepackage{verbatim}
\newcommand{\field}[1]{\mathbb{#1}}
\newcommand{\R}{\field{R}}

\newcommand{\Z}{\field{Z}}

\newtheorem{theorem}{Theorem}[section]
\newtheorem{lemma}[theorem]{Lemma}
\newtheorem{corollary}[theorem]{Corollary}
\newtheorem{proposition}[theorem]{Proposition}

\newtheorem{definition}[theorem]{Definition}

\theoremstyle{definition}

\newtheorem{remark}[theorem]{Remark}

\newcommand{\sign}{\mbox{sign}}

\begin{document}

\parskip 3pt

\title{Lie algebras of curves and loop-bundles on surfaces}

\author[J. Alonso]{Juan Alonso}\author[M.  Paternain]{Miguel Paternain}\author[J. Peraza]{Javier Peraza}

\address{\\Universidad de la Rep\'ublica, Centro de Matem\'atica,  Facultad de Ciencias, Igu\'a 4225, 11400 Montevideo
\\URUGUAY}
\email{juan@cmat.edu.uy}
\email{miguel@cmat.edu.uy}
\email{jperaza@cmat.edu.uy}

\author[M. Reisenberger]{Michael Reisenberger}

\address{\\Universidad de la Rep\'ublica, Instituto de F\'\i sica,  Facultad de Ciencias, Igu\'a 4225, 11400 Montevideo
\\
URUGUAY}

\email{miguel@fisica.edu.uy}

%\subjclass{}

\maketitle

\begin{abstract}
W. Goldman and V. Turaev defined a Lie bialgebra structure on the $\mathbb Z$-module generated by free homotopy classes of loops of an oriented surface (i.e. the conjugacy classes of its fundamental group). We develop a generalization of this construction replacing homotopies by thin homotopies, based on the combinatorial approach given by M. Chas. We use it to give a geometric proof of a characterization of simple curves in terms of the Goldman-Turaev bracket, which was conjectured by Chas.

\end{abstract}

\section{Introduction}

\label{sec:intro}

Goldman \cite{MR846929} and  Turaev \cite{MR1142906}  defined a Lie bialgebra structure on the $\mathbb Z$-module generated by
the free homotopy classes of loops of an oriented surface $M$ (i.e. the conjugacy classes of $\pi_1(M)$). The
bracket is defined by
\begin{equation}\label{goldman}
[X ,Y ]_{\pi_1(M)}=\sum_{p\in \alpha\cap \beta}\epsilon(p; \alpha, \beta) \{\alpha_p \beta_p\}    
\end{equation}
where $\alpha$ and $\beta$ are representatives of $X$ and $Y$
respectively, intersecting at most at  transversal double points, the
number $\epsilon(p; \alpha, \beta)=\pm 1$ denotes the oriented
intersection number of $\alpha$ and $\beta$ at $p$, and $ \{\alpha_p \beta_p\} $ is the conjugacy class of the element $\alpha_p\cdot \beta_p\in
\pi_1(M, p)$ where  $\alpha_p$ and $\beta_p$  are the elements
of $\pi_1(M, p)$ that correspond to reparametrize $\alpha$ and $\beta$ to start (and end) at $p$.

Turaev showed in \cite{MR1142906} that there is a colagebra
structure that gives rise to a Lie bialgebra. Chas in
\cite{MR2041630} proposed a combinatorial model for this bialgebra
structure. The aim of this paper is to develop a
generalization of the Goldman-Turaev construction replacing
homotopies by thin homotopies.

Let $M$ be an oriented surface endowed with any Riemannian metric. We denote by $\mathcal E(M)$ the set of classes of piecewise geodesic curves in $M$ modulo endpoint-preserving {\em thin homotopies}, which we shall define with precision in section \ref{s.prelim}. For each $x\in M$ we define ${\mathcal L}_x(M)$ as the elements of $\mathcal E(M)$ that start and end at $x$. It will be easy to notice that this is a group under concatenation, and $\pi_1(M,x)$ is a quotient of it.

 Let $S(M)$, the {\em space of strings} of $M$, be the set of conjugacy classes of
${\mathcal L}_x(M)$ in $\mathcal E(M)$, i.e. $g, h\in \mathcal L_x(M)$  are conjugate if there is $p\in \mathcal E(M)$ such that $p^{-1}gp=h$.  Let $\mathcal S(M)$  be the free abelian group generated by $S(M)$. In this paper we shall define a bracket $[ \; , \; ]$ on
$\mathcal S(M)$ following the lines of Chas in \cite{MR2041630}. Then we shall show

\begin{theorem}\label{jacobi}
$(\mathcal S(M), [\;,\;])$ is a Lie algebra.
\end{theorem}

It is also possible to give a coalgebra structure and show that $\mathcal S(M)$ is a Lie bi-algebra using the same techniques, though we will not present this construction explicitly. We will see that the Goldman-Turaev structure on $M$ is the quotient of $\mathcal S(M)$ obtained by taking regular (non-thin) homotpies. This fact will emerge naturally from our construction of the bracket, and the case for the co-bracket is analogous. The Goldman algebra is relevant in the study of spaces of representations of the fundamental groups of $2$-manifolds, which in turn, can be regarded as  moduli spaces of flat connections on orientable $2$-manifolds (\cite{jeffrey2005flat}). Our construction may also play a role in the study of the space of all connections but this aspect is not going to be treated here.

Chas and Krongold obtained an algebraic characterization of homotopy classes of simple curves in terms of the bracket  (\cite{chas2010algebraic}) and the cobracket (\cite{chas2016algebraic}) .
(See also \cite{chas2016extended} and  \cite{chas2021lie}).  In the present paper we give   a
geometric  proof of a different characterization of simple curves in terms of the bracket, which was conjectured by Chas in \cite{MR2041630}. For other related results see  \cite{chas2010self}, \cite{chas2012self} \cite{chas2015goldman}  \cite{chas2010minimal}, \cite{chas2019almost}, 
\cite{chas2015relations}.

A string $X$ is {\em primitive} if every representative $g\in \mathcal L_x(M)$ (for any $x\in M$) of $X$ is primitive in the sense of group theory: if there is no $h\in \mathcal L_x(M)$ such that $g=h^n$ with $n>1$. A piecewise geodesic closed curve will be called {\em simple} if it has no stable self-intersections, i.e. if there is a small perturbation that has no self-intersections (namely, that is simple in the usual sense).

We shall show the following theorem:

\begin{theorem}\label{simple}
Let $M$ be an oriented  surface. A primitive $X\in S(M)$ has a simple
representative if and only if $[X , X^{-1}]=0$.
\end{theorem}

The geometric group theory approach of our  proof of Theorem \ref{simple} may have some independent interest beyond the result itself.  The aforementioned proof includes some tools, namely, the notion of $\alpha$-oriented subgroups (see section \ref{alpha_oriented}),  which could be useful in other contexts. 

Theorem \ref{simple} will allow us to prove a conjecture posed by  Chas (\cite{MR2041630}):

\begin{corollary}
Let $M$ be a compact oriented surface with boundary. A primitive free
homotopy class $X$ of $M$ has a simple representative if and only if
$[X , X^{-1}]_{\pi_1(M)}=0$.
\end{corollary}

\begin{proof}
Let $p:\tilde M\to M$ be the universal covering of $M$. Let
$D\subset \tilde M$ be a fundamental polygon as in
\cite{MR2041630}.  Since  $M$ has boundary  we can choose on $\tilde M$ a metric of constant non positive curvature. 

The fundamental group is freely generated  by the
set $T$ of those $g\in \pi_1(M)$ such that $gD\cap D\neq \emptyset$. Let
$\mathcal L_{x,0}(M)\subset \mathcal L_x(M)$ be the subgroup of  $\mathcal L_x(M)$
generated by those elements of the form $h_g=[p\circ\alpha_g]$ where
$\alpha_g$ is a geodesic (corresponding to the chosen metric) joining $x$ to $gx$ for each $g\in T$. Denote by $S_0(M)$ the set of strings corresponding to the loops in $\mathcal L_{x,0}(M)$. Note that
$\pi_1(M)$ is isomorphic to $\mathcal L_{x, 0}(M)$ via the standard quotient, i.e. taking non-thin homotopies. As we mentioned below Theorem \ref{jacobi}, this quotient is a Lie algebra homomorphism, thus it gives an isomorphism between the subalgebra of $\mathcal S(M)$ generated by $S_0(M)$ and the Goldman-Turaev algebra on the free homotopy classes of $M$. Then we apply Theorem \ref{simple} to strings in $S_0(M)$ to conclude this proof.
\end{proof}

\section{The loop bundle}

\subsection{Definitions of thin homotopies and the spaces of loops} \label{s.prelim}

Let $I$ be the unit interval and $M$ a Riemannian manifold. We begin by recalling some standard notations. A {\em path} in $M$ is a continuous function from $I$ to $M$, and  we say that two paths  $a,b:I \to M$ are equivalent  modulo
reparametrization if there is an orientation preserving homeomorphism $\sigma:I\to I$ such that $a\circ \sigma=b$.  
 Denote by $\Omega_0$ the quotient set under this
equivalence relation. If $a(1)=b(0)$ we define $ab$ and $a^{-1}$ as
follows: $ab(t)=a(2t)$ if $t\in [0,1/2]$ and $ab(t)=b(2t-1)$ if
$t\in [1/2,1]$; $a^{-1}(t) = a(1-t)$ for every $t\in [0,1]$. Let
$e_x \in \Omega_0$ be the contant path at $x$, i.e. $e_x(t) =x$ for every $t\in [0,1]$.

In order to define what we call {\em thin homotpy} between piecewise geodesic paths we need to consider another preliminary equivalence, which amounts to collapse constant sub-paths. Let  $a$ be a non-constant path in $M$. We shall define a {\em minimal form} $a_r$ for $a$ as follows: let $I_i \subset I$ be the maximal subintervals in which
$a$ is constant, and let $\sigma: I \to I$ be a surjective non-decreasing continuous function, constant in each $I_i$ and strictly
increasing in $I - \bigcup_i I_i$. Then there is $a_r: I \to M$ such that $a =a_r \circ\sigma$, which is non-constant on any subinterval of $I$ (this map is obtained by a universal property of quotients). Different choices of the function $\sigma$ give rise to minimal forms that are equivalent modulo reparametrization, and moreover, if two paths $a$ and $b$ are equivalent, so are any of their minimal forms $a_r$ and $b_r$. This allows us to define the {\em minimal class} of an element of $\Omega_0$ (as the class of any minimal form of any representative), and take a quotient $\Omega_1$  where we identify two elements of $\Omega_0$ if they have the same minimal class (extending the definition to constant paths in the trivial way). The product and inverse are well defined on $\Omega_1$, and the classes of constant paths are units for the product.

%The above choices allow us to define a unique function $a_r: I \to M$ such that $a =a_r \circ\sigma$. Notice that $a_r$ is non-constant on any subinterval of $I$ and that its equivalence class under reparametrization depends only on $a$ (and not on the arbitrary choices available for $a$ and $\sigma$). We identify elements of $\Omega_0$ having the same reduced class and denote by $\Omega_1$ the space of equivalence classes.  

%Let $\Gamma\subset \Omega_1$ be the set of classes represented by  geodesics. 
 %Let $\Omega$ be the set of finite concatenations of curves in $\Gamma$. For the rest of the paper we say that elements $\alpha\in \Omega$ are curves, which shall be represented by greek letters.

Let $\Omega\subset \Omega_1$ be the set of classes of either constant paths or paths that are {\em piecewise geodesic}, i.e. a finite concatenation of geodesic segments. 
Notice that for $\alpha\in \Omega$ there are well defined notions of endpoints $\alpha(0)$ and $\alpha(1)$, of image $\alpha(I)$, and of length $l(\alpha)$. Throughout the paper we will refer to the elements $\alpha \in \Omega$ as {\em curves}, and say that $\alpha$ is a {\em closed curve} if $\alpha(0)=\alpha(1)$. %We shall use greek letters to denote curves.

In the set $\Omega$ we consider the equivalence
relation generated by the identifications $\alpha aa^{-1}\beta\sim
\alpha\beta$. This is what we call equivalence under {\em thin homotopies}. With the formal definition in hand, we recall the concepts from the introduction: Let $\mathcal E(M)$ denote the quotient set of $\Omega$ under thin homotopies, and 
 let ${\mathcal L}_x(M)$ be the projection onto $\mathcal E(M)$ of the set of closed curves starting and ending at $x$. Note that
${\mathcal L}_x(M)$ is a group under concatenation 
whose identity element, $id_x$, is the equivalence class of $e_x$,
the constant path at $x$.

%%%%%%%%%%%%%%%%%5
\subsection{Reductions and basic properties}

A {\em reduction} for $\alpha\in\Omega$ is a factorization of the form $\alpha=acc^{-1}d$ with nontrivial $c$. We say that $\alpha$ is {\em reduced} if it admits no such reduction. Since the curves in $\Omega$ are classes of piecewise geodesic (or constant) paths, it is easy to show that every element of $\mathcal E(M)$ has a unique reduced representative in $\Omega$ (though the proof of uniqueness may be a bit cumbersome). The {\em reduced form} of $\alpha\in\Omega$ is the unique reduced curve that is equivalent to $\alpha$ under thin homotopy.

Next we point out some basic facts that will be used without explicit reference throughout the article. Firstly notice that a curve $\gamma\in\Omega$ satisfies $\gamma=\gamma^{-1}$ only when it is of the form $\gamma=cc^{-1}$, thus for a reduced curve $\gamma$  it happens only when $\gamma$ is constant. Next we see that the concept of length in $\Omega$ satisfies the expected properties, namely that:
\begin{itemize}
\item $l(ab)=l(a)+l(b)$, and
\item if $ab=cd$ with $l(a)=l(c)$, then $a=c$ and $b=d$.
\end{itemize}

For $\gamma,\,\delta\in \Omega$ we shall write $\gamma\subset\delta$ if we have $\delta=a\gamma b$ for $a,b\in\Omega$. In case $a$ is trivial we say that $\gamma$ is an {\em initial segment} of $\delta$, and if $b$ is trivial that $\gamma$ is a {\em final segment} of $\delta$. Note that if $\delta$ is reduced, so must be $\gamma$.

We say that two curves $\gamma$ and $\delta$ {\em overlap} if an initial segment of one of them agrees with a final segment of the other, i.e. if we can write either $\gamma=ab$, $\delta=bc$ with $b$ non-constant, or $\gamma=ab$, $\delta=ca$ with $a$ non-constant. Note that if $\gamma$ is reduced and non constant, then $\gamma$ and $\gamma^{-1}$ cannot overlap: for instance, if $\gamma=ab$ and $\gamma^{-1}=bc$, we get $b=b^{-1}$ where $b$ is reduced, so $b$ must be constant.

\subsection{Definitions of loop bundle and horizontal lift}

Consider the space $\mathcal E(M)$ defined in the previous section, and let $[\alpha]\in \mathcal E(M)$ stand for the equivalence class of $\alpha\in \Omega$. Let
$\mathcal E_x(M)$ be the set of the $[\alpha]\in \mathcal E(M)$ such that $\alpha(0)=x$; define $\pi:\mathcal E_x(M)\to M$ by
$\pi([\alpha])=\alpha(1)$, and observe that $\mathcal L_x(M)=\pi^{-1}(x)$. The group $\mathcal L_x(M)$ acts on $\mathcal E_x(M)$ by left multiplication and for all $[\alpha]\in \mathcal L_x(M)$ and $[\gamma]\in
\mathcal E_x(M)$ we have $\pi\big([\alpha][\gamma]\big)=
\pi([\gamma])$; hence the quadruple $(\mathcal E_x(M) , \mathcal
L_x(M), M,\pi) $ is a principal fiber bundle over $M$, with
structure group $\mathcal L_x(M)$.

Let $\gamma$ be a path in $M$, and take $p\in\mathcal E_x(M)$ with
$\pi(p)=\gamma(0)$. We define the {\em horizontal lift} of $\gamma$
at $p$ to be the path $\tilde{\gamma}$ in $\mathcal E_x(M)$ which
is obtained in the following way. Take $\beta$ any representative of
$p$ (i.e. $p=[\beta]$), and for each $s\in I$ set $\gamma_s$ to be
the path in $M$ defined by $\gamma_s(t)=\gamma(st)$. Then
$\tilde{\gamma}(s)=[\beta\gamma_s]$. This horizontal lift can be seen as a topological 
connection in the bundle $(\mathcal E_x (M), \mathcal L_x(M),
M,\pi)$. %%
 We will say that a path in $\mathcal
E_x(M)$ is {\em horizontal} if it can be obtained by horizontal lift (see
\cite{MR0077122}, \cite{MR0124068}). Note that the concept of horizontal lift is well defined at the level of curves (i.e. in $\Omega_1$), thus we may speak of {\em horizontal curves}. 

We define the {\em length of an horizontal curve} as the length of the projection. Observe that the action of $\mathcal L_x(M)$ preserves the set of horizontal curves, as well as their lengths (by definition). We should clarify that we are not giving a metric on $\mathcal E_x(M)$.

\subsection{Conjugacy classes in ${\mathcal L}_x$ and the space of strings.}

Recall that the {\em space of strings} $S(M)$ is the set of conjugacy classes of
${\mathcal L}_x(M)$ in $\mathcal E(M)$, i.e. $g, h\in \mathcal L_x(M)$  are conjugate if there is $p\in \mathcal E(M)$ such that $p^{-1}gp=h$. This does not depend on $x$ because of the following remark.

%Finally, $\mathcal S(M)$  is the free abelian group generated by $S(M)$. Note that a string $X\in S(M)$ is {\em primitive} if for every  closed $\alpha\in\Omega$ representing $X$ there is no $\gamma$ such that $\alpha=\gamma^n$ with $n>1$.

%---------

%Let $S(M)$, the space of strings, be the set of conjugacy classes of
%${\mathcal L}_x(M)$ in $\mathcal E(M)$. Note this does not depend on $x$ because of the following remark.
\begin{remark}\label{bpoint} (Change of basepoint)
If $x,y\in M$ and $\gamma_0:I\to M$ has
$\gamma_0(0)=x$ and $\gamma_0(1)=y$, let $p_0=[\gamma_0]$ and define
the maps $\psi: \mathcal E_x(M) \to \mathcal E_y(M)$ by
$\psi(p)=p_0^{-1}p$ and $\phi: \mathcal L_x(M)  \to \mathcal L_y(M)$ by $\phi(g)=p_0^{-1}gp_0$. Then $\phi$ is an isomorphism of groups, and $(\psi,\phi)$ is an isomorphism of fiber bundles over $M$, commuting with the horizontal lift.
\end{remark}

We say  that a closed curve $\alpha$  is {\em cyclically reduced}
if  $\alpha$  is reduced and it cannot be factorized as $cac^{-1}$
 with non trivial $c$. For a cyclically reduced curve $\gamma$, we say that $\beta\in\Omega$ is a {\em permutation} (or {\em cyclical permutation}) of $\gamma$ if there are $r,s\in\Omega$ such that $\gamma=rs$ and $\beta=sr$. If $s$ and $r$ are non constant we say that $\beta$ is a {\em non trivial  permutation} of $\gamma$. Note that permutation is an equivalence relation among the cyclically reduced curves in $\Omega$.

For a string $X\in S(M)$ we can find $x\in M$ and a cyclically
reduced curve $\alpha$ based at $x$, such that $X$ is the conjugacy class of $[\alpha]\in\mathcal L_x(M)$. On the other hand, permutation agrees with conjugacy in $\mathcal E_x(M)$ among cyclically reduced curves, therefore we have

%Let $\gamma$ be a cyclically reduced curve. We say that $\beta$ is a permutation of $\gamma$ if $\gamma=rs$ and $\beta= sr$.If  $s$ and $r$ are non constant we say that $\beta$ is a non trivial  permutation of $\gamma$ 

% Note that permutation is an equivalence relation among the cyclically reduced curves in $\Omega$. 

\begin{remark}
 There is a bijection between $S(M)$
and the permutation classes of cyclically reduced curves.
\end{remark}

Throughout the paper, when we take representatives of strings we will always assume them to be cyclically reduced. If $X\in S(M)$ and $\alpha$ is a representative of it, we define $X^{-1}$ as the permutation class of $\alpha^{-1}$. It is straightforward to check that if $\alpha$ is non-constant then $\alpha^{-1}$ is not a permutation of $\alpha$. Thus $X\neq X^{-1}$ unless $X$ is trivial.

 A cyclically reduced curve $\alpha$ is {\em primitive} if there is no $\gamma\in\Omega$ such that $\alpha=\gamma^n$ with $n>1$.  The following easy result is well known.

\begin{lemma}\label{primitive}
Let $\alpha$ be a cyclically reduced curve. Then $\alpha$ is not primitive if and only if $\alpha$ has a non trivial permutation $\hat\alpha$ such that
$\alpha=\hat \alpha$. 
\end{lemma}

Note that a string $X\in S(M)$ is primitive, as defined in the introduction, if it has a cyclically reduced representative that is primitive.

\section{Lie bialgebra structure}

\subsection{Linked pairs} \label{LPs}

In order to define the bracket in $\mathcal S(M)$ we need a way of encoding the intersections of curves in $\Omega$ that are stable under local homotopy. We do this by adapting the notion of {\em linked pairs} from Chas \cite{MR2041630} to our context.

Let $\alpha_1, \alpha_2$ and $\gamma$ in $\Omega$ be classes of geodesic segments  contained in a normal ball such that $\alpha_1(1)=\alpha_2(0)=:y$ and either $\gamma(0)=y$ or $\gamma(1)=y$. Assume that $\alpha=\alpha_1\alpha_2$ is reduced and $\gamma$ only meets $\alpha$ at $y$. Take $\rho>0$ small enough so that $B(y, \rho)$ is a normal ball and  $\alpha_1$, $\alpha_2$ and $\gamma$ are not contained in it, and let  $z_1, z_2, z$ be the intersection points of 
$\alpha_1$, $\alpha_2$ and $\gamma$ with  $\partial B(y, \rho)$ respectively.  Since $M$ is oriented, the orientation of $B(y, \rho)$ induces an orientation of $\partial B(y, \rho)\cong S^1$, which is equivalent to giving a circular order on $\partial B(y, \rho)$. 

We write $sign(\alpha, \gamma)=1$ if  either  $\gamma(0)=y$  and  the order of the sequence $z_2, z, z_1$ coincides with the circular order  of $\partial B(y, \rho)$, or $\gamma(1)=y$ 
and  the order of the sequence $z_2, z_1, z$ coincides with the circular order of $\partial B(y, \rho)$.  Otherwise we write  $sign(\alpha, \gamma)=-1$.
Notice that this sign does not depend on the choice of $\rho$. 

Informally one could say that $sign(\alpha, \gamma)=1$ if $\gamma$ is either outgoing at the ``left'' side of $\alpha$ or incoming at the ``right'' side of $\alpha$, while $sign(\alpha, \gamma)=-1$ if one of the reverse situations happens.

%Let $B(y, \rho)$ be a normal ball of radius $\rho$ about a point $y$ in $M$.  Since $M$ is oriented, the orientation of $B(y, \rho)$ induces an orientation of $\partial B(y, \rho)$. 

%Let $\alpha_1, \alpha_2$ and $\gamma$ in $\Omega$ be classes of geodesic segments  contained in a normal ball so that $\alpha=\alpha_1\alpha_2$ is  reduced and such that  $\gamma$ meets $\alpha$ at $\gamma(0)=\alpha_1(1)$ or it meets $\alpha$  at $\gamma(1)=\alpha_1(1)$. 
%Let $\rho$ be such that $\alpha_1$, $\alpha_2$ and $\gamma$ are not contained in $B(\alpha_1(1), \rho)$. Let $z_1, z_2, z$ be the intersection points of  $\alpha_1$, $\alpha_2$ and $\gamma$ with  $\partial B(\alpha_1(1), \rho)$ respectively.  

Since elements of $\Omega$ are piecewise geodesic curves, the intersections between two elements are either transversal or along an interval. Taking this into account, we discuss the general forms of these intersections and indicate which ones will constitute linked pairs. A factorization of a curve $\alpha\in \Omega$ is a sequence $(\alpha_1,\ldots,
\alpha_n)$ such that $\alpha=\alpha_1\cdots\alpha_n$ where
$\alpha_i\in \Omega$. 

\begin{definition} \label{def:lp}
Consider the following factorizations (of some curves)
\[ A=(a, \eta, b) \]
\[B= (c, \xi,  d) \]

where $a,b,c,d$ are geodesics contained in normal balls. We say that
$(A, B)$ is a {\em linked pair} if any of the following conditions hold

\begin{enumerate}
\item
$\eta=\xi=\mbox{point}$, $d$ meets $ab$ only at $d(0)$ and $c$ meets $ab$ only at $c(1)$, and 

\[sign(ab, d)=sign(ab, c)\] 
\item $\eta=\xi$ (non constant), if we factorize $\eta=\gamma_1 \eta_1\gamma_2$ such that $\gamma_1$ and $\gamma_2$ are contained in normal balls, we have that $d$ meets $\gamma_2 b$ only at $d(0)$ and $c$ meets $a\gamma_1$ only at $c(1)$, and 
\[sign(\gamma_2b, d)= sign(a\gamma_1, c)  \]

\item
$\eta=\xi^{-1}$ (non constant), if we factorize $\eta=\gamma_1 \eta_1\gamma_2$ such that $\gamma_1$ and $\gamma_2$ are contained in normal balls, we have that $c$ meets  $\gamma_2 b$ only at $c(1)$ and $d$ meets 
$a\gamma_1$ only at $d(0)$, and 
\[sign(a\gamma_1, d)= sign(\gamma_2b, c)  \]
\end{enumerate}
\end{definition}

We define the sign of the linked pair as follows: In case (1) we set $sign(A, B)=sign(ab, d)$, in case (2) we set $sign(A, B)=sign(\gamma_2b, d)$ and in case (3) set $sign(A, B)=sign(a\gamma_1, d)$.

If all but the orientation (sign) conditions hold we say that $(A,B)$ is an {\em intersection pair}. Notice that the intersections between two cyclically reduced curves in $\Omega$ can locally be written in the form of intersection pairs. The orientation conditions say that an intersection pair $(A,B)$ is a linked pair exactly when the intersection between the underlying curves is stable under small perturbations. We shall refer to linked pairs of {\em type} (1), (2) or (3) according to which one of the conditions they satisfy in Definition \ref{def:lp}, and we do the same for intersection pairs. 

Next we turn to the intersections of cyclically reduced curves in $\Omega$ in a global sense, i.e. in a way that takes account of multiplicities. 

If $\alpha$ is a closed curve, we say that $P=(\xi,\eta)$ is a {\em cyclic factorization} of $\alpha$ if either $\alpha=\xi_1 \eta \xi_2$ with $\xi = \xi_2 \xi_1$ or $\alpha = \eta_1 \xi \eta_2$ with $\eta = \eta_2 \eta_1$. There is a slight abuse of notation here, as the decompositions $\xi = \xi_2 \xi_1$ or $\eta = \eta_2 \eta_1$ are needed for recovering $\alpha$, and are indeed intended as part of the definition, though we drop them from the notation to make it less cumbersome. If $\alpha$ is cyclically reduced and $\beta$ is a permutation of it, then there is a bijection between  the cyclic factorizations of $\alpha$ and those of $\beta$. Notice though, that in order to talk about cyclic factorizations of a string, we need to choose a curve representative first. The reason for this choice of definition is that we want to keep track of the position of the sub-curves ($\xi$ and $\eta$) with respect to a chosen parameter-basepoint (which is well defined in $\Omega$). This detail will make a difference for strings that are not primitive.   

%We consider the position of $\eta$ inside of $\alpha$ (i.e. the parameter $t$ such that $\eta(0)=\alpha(t)$) as part of the information defining a cyclic factorization. So, we count as different two cyclic factorizations of $\alpha$ that differ in their positions in $\alpha$.

%Note that a change of parametrization induces a natural bijection between the cyclic factorizations of the corresponding curves. Also, if $\beta$ differs from $\alpha$ in a change of base point, there is a bijection between the cyclic factorizations of $\alpha$ and those of $\beta$.
\begin{definition} \label{def:lp cyclic}
Let $\alpha,\beta\in \Omega$ be cyclically reduced closed curves. A {\em linked pair between
$\alpha$ and $\beta$} is a pair $(P, Q)$ of cyclic factorizations
$P=(\alpha_1,\eta)$ and $Q=(\beta_1, \xi)$ of $\alpha$ and $\beta$ respectively, such that if we write
\begin{itemize}
\item $\alpha_1=b \hat
\alpha_1 a$ and $\beta_1= d \hat \beta_1 c$, where $a,b,c,d$ are geodesics contained in normal balls, and
\item $A=(a,\eta, b)$ and $B=(c,\xi, d)$, 
\end{itemize}
then $(A,B)$ is a linked pair.

\end{definition}

Notice that the concatenations $a\eta b$ and $c\xi d$ are well defined, so the above definition makes sense. Moreover, they are sub-curves of some permutations of $\alpha$ and $\beta$ respectively, thus saying that $(A,B)$ is a linked pair means that there is a stable intersection between $\alpha$ and $\beta$, or the strings they represent. Defining $P$ and $Q$ as cyclic factorizations keeps track of the position of the intersection segments relative to the parameter-basepoints of $\alpha$ and $\beta$, so intersections that repeat count as different linked pairs. This amounts to counting multiplicity, just as is usual in differential topology for the intersection between transversal smooth paths. Notice also that taking a permutation of $\alpha$ or $\beta$ induces a natural bijection between the sets of linked pairs. 

We define the {\em length} of a linked pair as $l(P,Q)=l(\eta)=l(\xi)$, i.e. as the length of the intersection segment. The {\em type} of $(P,Q)$ shall be the type of $(A,B)$ in Definition \ref{def:lp cyclic}. % recalling that length is well defined in $\Omega$. 

 %$\alpha_1=b \hat\alpha_1 a$, $\beta_1= d \hat \beta_1 c$ and (A,B)  is a linkedpair, where  $A=(a,\eta, b)$ and $B=(c,\xi, d)$. Set $l(P,Q)=l(\eta)=l(\xi)$, where $l(\eta)$ is the length of $\eta$.

\subsection{Definition of the string bracket}

In this section we define the bracket following closely the
presentation in  \cite{MR2041630}. Since the definition of the
co-bracket involves no new ideas we omit it. Recall that  $S(M)$,
the space of strings, is the set of conjugacy classes of ${\mathcal
L}_x(M)$. Also recall that  $\mathcal S(M)$  is the free abelian group generated by
$S(M)$, in which we shall define the bracket.

For $X\in S(M)$ and an integer $n>0$, let $X^n$ be the conjugacy class of $[\alpha]^n$, where $[\alpha]$ represents $X$. Define also $l(X)=l(\alpha)$ where $\alpha$ is a cyclically reduced representative of $X$, noting that different choices of such representative have the same length. Since $\alpha$ is cyclically reduced, we have that $l(X^n)=nl(X)$.

Although we will not present the definition of the co-bracket, we give the main definition in which it is based, for the sake of completeness. 

\begin{definition} \label{lp1} Let $X$ be a string. We define $LP_1(X)$, the set of linked pairs
 of $X$, as the set of linked pairs of any two representatives of $X$.
\end{definition}

By the discussion at the end of the previous section, the choice of representatives of $X$ in Definition \ref{lp1} does not affect $LP_1(X)$. It is possible to show, similarly as in \cite{MR2041630}, that this definition reflects the stable self-intersections of $X$, at least when $X$ is primitive, in a 2 to 1 correspondence: each stable self-intersection corresponds to two linked pairs of the form $(P,Q)$ and $(Q,P)$. Non-primitive closed curves have stable self-intersections, in the sense of the self-intersections of a transversal perturbation, that do not arise from linked pairs. Since we will not focus on the co-bracket, we shall not prove these assertions. Next we turn to the case of linked pairs between two strings, that will be the key for the construction of the bracket. 

%By the discussion following the definition of linked pairs, $LP_1(X)$ reflects the stable self-intersections of $X$, at least when $X$ is primitive. Non-primitive closed curves have stable self-intersections, in the sense of the self-intersections of a transversal perturbation, that do not arise from linked pairs. Next we turn to the case of the stable intersections of two strings.

\begin{definition} \label{lp2} Let $X$ and $Y$ be strings. Define $LP_2(X,Y)$, the set of linked pairs of $X$ and $Y$, as the set of linked pairs $(P,Q)$   between representatives of  $X^n$ and $Y^m$ for $n,m\geq 1$, where
$l(X^{n-1})\leq l(P, Q)< l(X^{n})$ and $l(Y^{m-1})\leq l(P, Q)< l(Y^{m})$. (With the convention that $l(X^0)=l(Y^0)=0$).
\end{definition}

Again, different choices of representatives for the strings in Definition \ref{lp2} yield sets $LP_2(X,Y)$ that are in natural bijection.
 
\begin{remark} \label{r:lp2} The powers are necessary:
Consider $\alpha$ and $\beta$, closed geodesics starting and ending at the same point $x$ and  meeting transversally at $x$.   Let $X$ be the conjugacy class of $[\alpha]$ and $Y$ the conjugacy class of $[\beta] [\alpha^2]$. 
There is no linked pair between $X$ and $Y$ but there is a linked pair between $X^3$ and $Y$. Note that the core segment of the linked pair is $\alpha^2$.

\end{remark}

On the other hand, it can be shown that $LP_2(X,X)=LP_1(X)$, i.e. the powers are not needed in the case $X=Y$. We shall see later, in Lemma \ref{chas20}, that $LP_2(X,Y)$ captures the notion of stable intersections between $X$ and $Y$. It is not inmediate from Definition \ref{lp2} that $LP_2(X,Y)$ is finite, the proof of this fact will be based in the following result.

%Through the following proposition, given  strings $X$ and $Y$ we can find a finite set of curves giving an alphabet so that representatives of $X$, $Y$ and the curves in their factorizations, can be obtained as words in this alphabet.

\begin{proposition} \label{alphabet} Let $\mathcal{U}=\{\alpha_1,\ldots,\alpha_n \}$ be a finite set of piecewise geodesic curves. There are
factorizations $\alpha_i=a_{i,1}\cdots a_{i,n_i}$ such that whenever ${a_{i,j}\cap a_{k,l} \neq \emptyset}$, either

\begin{enumerate}

\item $a_{i,j}$ and $a_{k,l}$ meet only at one endpoint.

\item $a_{i,j}=a_{k,l}$

\item $a_{i,j}=a_{k,l}^{-1}$

\end{enumerate}

\end{proposition}

\begin{proof} Subdivide any factorization of the curves $\alpha_i$ until the desired properties are obtained. This will happen because
of the transversality properties of the geodesics.

\end{proof}

Given strings $X$ and $Y$, Proposition \ref{alphabet} allows us to find a finite set of curves that works as an alphabet for writing some representatives of $X$ and $Y$, as well as all the core curves of the intersection pairs between (powers of) these representatives. Thus we can write the cyclic factorizations that make up the elements of $LP_2(X,Y)$ as words in this alphabet. 

%so that representatives of $X$, $Y$ and the curves in their factorizations, can be obtained as words in this alphabet.

%With this proposition the proof of the following lemma is a straightforward adaptation of Lemma 2.9 of \cite{MR2041630}

\begin{lemma} For any strings $X$ and $Y$, $LP_2(X,Y)$ is finite.
\end{lemma}

\begin{proof} Using Proposition \ref{alphabet} as indicated above, this becomes a straightforward adaptation of Lemma 2.9 of \cite{MR2041630}.

\end{proof}

\begin{definition} \label{dotproduct} Let $X$ and $Y$ be strings and
$(P,Q)\in LP_2(X,Y)$. Write   $P=(\alpha_1, \eta)$ and $Q=(\beta_1, \xi)$ as in the definition of linked pairs, and let $\alpha$ and $\beta$ be the representatives of $X$ and $Y$ that satisfy the following:

\begin{itemize}
\item If $(P, Q)$ is of type (1) or (2), then \[ \alpha^n=\alpha_1\eta \mbox{ and } \beta^m=\beta_1\xi \]
(where $n,m\geq 1$ are the powers of $X$ and
$Y$ that correspond to $(P,Q)$ in Definition \ref{lp2}).

\item If $(P, Q)$ is of type (3), then 
\[ \alpha^n=\alpha_1\eta  \mbox{ and } \beta^m=\xi\beta_1 \]
(for the same $n,m\geq 1$).
\end{itemize}  

In any of the above cases, define $(X\cdot_{(P,Q)} Y)$ to be the conjugacy class of $[\alpha][\beta]$.

%.  First assume that $(P, Q)$ satisfies condition (1) or (2) in the definition of linked pairs. 
%Let $\alpha$ represent  $X$ and $\beta$ represent $Y$ such that $\alpha\alpha_2=\alpha_1\eta$ and $\beta\beta_2=\beta_1\xi$($\alpha_2$ and $\beta_2$ may be non trivial because we are taking powers of $X$ and $Y$).
%Assume that $(P, Q)$ satisfies condition (3) in the definition of linked pairs. Let $\alpha$ represent  $X$ and $\beta$ represent $Y$ such that $\alpha \alpha_2=\alpha_1\eta$ and $\beta\beta_2=\xi^{-1}\beta_1$.  In any of the above cases  define $(X\cdot_{(P,Q)} Y)$ to be the conjugacy class of $[\alpha][\beta]$.

\end{definition}

We say that $(X\cdot_{(P,Q)} Y)$ is the {\em dot product} of $X$ and
$Y$ at $(P,Q)$. Notice that $\alpha$ and $\beta$ are the representatives of $X$ and $Y$ that we get by choosing parameter-basepoints at the ``ending'' of the linked pair's core curve. They are indeed loops based at the same point, so the concatenation $[\alpha][\beta]$ is well defined. 

\begin{definition} \label{d:bracket} Let $X$ and $Y$ be  strings, we define their bracket as
\[ [X,Y]=\sum_{(P,Q)\in LP_2(X,Y)}\sign(P,Q)(X\cdot_{(P,Q)} Y) \]

Then we extend the definition to $\mathcal{S}(M)$ so that  the bracket is bilinear. 
\end{definition}

\subsection{Linked pairs and differentiable curves}\label{dif}
Our next goal is to show the correspondence between linked pairs and stable intersections, which will lead to the relationship between the bracket in Definition \ref{d:bracket} and the Goldman-Turaev bracket given by equation \eqref{goldman}. This will in turn allow us to prove Theorem \ref{jacobi}, i.e. that Definition \ref{d:bracket} gives a Lie algebra.  

Let $C$ be a compact one dimensional complex on an oriented Riemannian surface $M$ whose edges are geodesic arcs, and consider a basepoint $x\in C$. Let $\tilde C$ be the universal covering of $C$. Then the following lemma is straightforward.

\begin{lemma}\label{lem:ident}
${\mathcal L}_x(C)\cong \pi_1(C, x)$ and ${\mathcal E}_x(C)\cong \tilde C$.

Moreover, these correspondences give an isomorphism of fiber bundles $$({\mathcal E}_x(C),{\mathcal L}_x(C),C,\pi) \cong (\tilde C,\pi_1(C, x),C,\pi)$$ where $\pi_1(C, x)$ acts on $\tilde C$ by deck transformations.

%${\mathcal L}_x(C)\cong \pi_1(C, x)$\\
%${\mathcal E}_x(C)\cong \tilde C$\\

\end{lemma}

Let $S(C)$ be the set of strings contained in $C$, and note that the string
bracket of Definition \ref{d:bracket} can be restricted to
$\mathcal S(C)$, the free abelian group on $S(C)$. We will denote this bracket by $[\;,\;]_C$.

For any set $V\subset M$ let  $V_{\varepsilon}$ denote an
$\varepsilon$-neighborhood of $V$. The following lemma is well
known, see \cite{MR621981} for a very general construction.

\begin{lemma}\label{tubular}
If $\varepsilon$ is small enough, there is a retraction
$\chi:C_{\varepsilon}\to C$. It induces an isomorphism $\chi_*:\pi_1(C_{\varepsilon},x)\to\pi_1(C,x)$ for any $x\in C$.
\end{lemma}

Let ${\mathcal S}_1(M)$ be the free abelian group generated by the
set of conjugacy classes of $\pi_1(M)$, where the Goldman-Turaev bracket is defined. By Lemma \ref{lem:ident} we see that $\chi$ induces an isomorphism of abelian groups   $\chi_*:{\mathcal S}_1(C_{\varepsilon})\to {\mathcal S}(C) $. We shall prove that $\chi_*$ sends the Goldman-Turaev bracket of the surface $C_{\varepsilon}$ to $[\;,\;]_C$.

Proposition 3.11 in \cite{MR2041630} can be rephrased as

\begin{lemma}\label{chas0} Let $[\alpha]$ and $[\beta]$ be representatives of strings in $S(C)$. Then there are differentiable curves $\gamma$ and $\delta$ in $C_{\varepsilon}$  such that 

\begin{itemize}
\item $\gamma$ and $\delta$ are $\varepsilon$-perturbations of $\alpha$ and $\beta$ respectively, 
\item $\alpha$ and $\beta$ are the reduced forms of $\chi\circ \gamma$ and $\chi\circ \delta$ respectively, and
\item $\gamma$ and $\delta$ intersect transversally, in at most double points, and determine no bigons.
\end{itemize}

Moreover, $\gamma$ and $\delta$ have minimal intersection among the curves in their homotopy classes.
\end{lemma}

We remark that in Lemma \ref{chas0} the curves $\chi\circ \gamma$ and $\chi\circ \delta$ need not be reduced, but the curves removed in their reduction have length less than $\varepsilon$, and each one is contained in some geodesic edge of the complex $C$, assuming $\varepsilon$ is small enough. Now we can relate the linked pairs of two strings in $S(C)$ with the intersections that are stable under perturbation in $C_{\varepsilon}$ .

\begin{lemma}\label{chas20} Let $[\alpha]$ and $[\beta]$ be representatives of strings $X$ and $Y$ in $S(C)$, and let $\gamma$ and $\delta$ be curves given by Lemma \ref{chas0}.

Then for each intersection point $p$  of $\gamma$ and $\delta$ there are $n,m\geq 1$, curves  $a\subset\gamma^n$ and $b\subset\delta^m$ meeting at $p$, and a linked pair $(P,Q)$ between $\alpha^n$ and $\beta^m$ that satisfy the following:

If $P=(\alpha_1,\xi)$ and $Q=(\beta_1,\eta)$, then the reduced forms of $\chi\circ a$ and $\chi\circ b$ can be written as $a_1\xi a_2$ and $b_1\eta b_2$ respectively, where $a_1,a_2,b_1,b_2$ are geodesic segments.

%\begin{itemize}
%\item if $P=(\alpha_1,\xi)$, then the reduced form of  $\chi\circ a$ can be written as $a_1\xi a_2$ where $a_1,a_2$ are geodesic segments, and  
%\item if $Q=(\beta_1,\eta)$, then the reduced form of  $\chi\circ b$ can be written as $b_1\eta b_2$ where $b_1,b_2$ are geodesic segments.
%\end{itemize}

Moreover, this correspondence is a bijection between the intersection points of $\gamma$ and $\delta$, and $LP_2(X,Y)$.
\end{lemma}

\begin{proof}
This is done with the  techniques of \cite{MR2041630}. We
consider $\pi:\tilde{C}_{\varepsilon}\to C_{\varepsilon}$, the universal cover of
$C_{\varepsilon}$. Then $\tilde{C}$ is embedded in
$\tilde{C}_{\varepsilon}$ as a tree, made of piecewise geodesic curves,
since $C$ is piecewise geodesic. Consider also
$\tilde{\chi}:\tilde{C}_{\varepsilon}\to\tilde{C}$ the lift of $\chi$, and
for each geodesic arc $c$ in the decomposition of $\tilde{C}$ let
$V(c)=\tilde{\chi}^{-1}(c)$. Then the sets $V(c)$ are homeomorphic to closed disks, they cover $\tilde{C}_{\varepsilon}$, and their interiors are disjoint. 

Now, for each $p$ in the intersection between $\gamma$ and $\delta$, we pick $\hat{p}\in\tilde{C}_{\varepsilon}$ projecting to $p$, and consider $\hat{\gamma}$ and $\hat{\delta}$ the infinite lifts of $\gamma$ and $\delta$ that meet only at $\hat{p}$ (recalling that $\gamma$ and $\delta$ have no bigons). Let $\hat \alpha$ and $\hat \beta$ be the respective reductions of $\tilde \chi \circ \hat \gamma$ and $\tilde \chi \circ \hat \delta$, which are infinite lifts of $\alpha$ and $\beta$. The set $\hat \alpha(I) \cap \hat \beta(I)$ is compact, since $\hat{\gamma}$ and $\hat{\delta}$ meet only once and are lifts of closed curves, and it is an arc, since $\hat \alpha$ and $\hat \beta$ are reduced. We shall write $\hat \alpha(I) \cap \hat \beta(I)$ as a curve in two ways, with possibly different orientations: we call $\hat\xi$ and $\hat\eta$ to the curves spanning $\hat \alpha(I) \cap \hat \beta(I)$ with the orientations given by $\hat \alpha$ and $\hat \beta$ respectively. It may be the case that $\hat \alpha(I) \cap \hat \beta(I)$ is just a point, then $\hat \xi$ and $\hat \eta$ will be constant (this will result in a linked pair of type (1) ).

% + Let $\hat\xi$ and $\hat\eta$ be the curves spanning $\hat \alpha \cap \hat \beta$ with the orientations given by $\hat \alpha$ and $\hat \beta$ respectively. 

% - and let $\hat a\subset\hat\gamma$ and $\hat b \subset \hat \delta$ be the arcs that verify $[\hat \xi]=[\hat\chi\circ\hat a]$ and $[\hat \eta]=[\hat\chi\circ\hat b]$. 

% + Since $\hat a$ is compact, it is contained in some lift of $\gamma^n$ inside $\hat \gamma$, for some power $n\geq 1$. We pick the minimum $n$ so that this inclusion is strict. The same argument yields $m$, minimum so that $\hat b$ is strictly contained in a lift of $\delta^m$ inside $\hat \delta$.

% - Thus we may write $\omega=\hat \alpha \cap \hat \beta$ as a piecewise geodesic arc $\omega=a_1\cdots a_k$ in $\tilde{C}$, choosing an orientation for $\omega$. 

Let
\[ V = \cup \{V(c): c\cap \hat \alpha(I) \cap \hat \beta(I) \neq \emptyset \} \]
Then $V$ is, topologically, a closed disk, and we have $\hat \alpha \cap V = c_1 \hat\xi c_2$ and $\hat \beta \cap V = d_1 \hat\eta d_2$, where $c_1,c_2,d_1,d_2$ are pairwise different geodesic segments.

For each geodesic segment $c$ that meets $\hat \alpha(I) \cap \hat \beta(I)$ but is not contained in it, we define the set $B(c)=\partial V(c)\cap\partial V$. Note that $B(c)$ is a segment in $\partial V \cong S^1$, and that $c$ has an endpoint in $B(c)$ and the other in $\hat \alpha(I) \cap \hat \beta(I)$. The segments $B(c)$ just defined are pairwise disjoint, and their union is the relative boundary of $V$ in $\tilde{C}_{\varepsilon}$. Taking $\varepsilon$ small enough, we may assume that the $\varepsilon$-neighborhood of $\hat \alpha$ only meets the relative boundary of $V$ at the arcs $B(c_1)$ and $B(c_2)$. Thus $\hat \gamma \cap V$ is an arc that enters $V$ through $B(c_1)$ and exits through $B(c_2)$, meeting no other segment of the relative boundary of $V$. Similarly we get that $\hat \delta\cap V$ is an arc that traverses $V$ from $B(d_1)$ to $B(d_2)$. 

Let $\hat a = \hat \gamma \cap V$ and $\hat b = \hat \delta \cap V$. They must intersect at $\hat p$, in particular $\hat p \in V$, since the complementary arcs of $\hat \gamma$ and $\hat \delta$ are in different components of $\tilde{C}_{\varepsilon}-V$. Note that $\tilde\chi\circ\hat a$ and $\tilde\chi\circ\hat b$ can be reduced, respectively, to $\hat \alpha \cap V = c_1 \hat\xi c_2$ and $\hat \beta \cap V = d_1 \hat\eta d_2$. We define the curves $a,b,\xi,\eta$  in the statement  as the respective projections of $\hat a, \hat b, \hat \xi, \hat \eta$.

By compactness, there are $n,m\geq 1$ such that $\hat a$ and $\hat b$ are contained in lifts of $\gamma^n$ and $\delta^m$ inside $\hat \gamma$ and $\hat \delta$ respectively. We choose $n, m$ minimal for these inclusions to be strict. Thus $\xi$ and $\eta$ induce cyclical factorizations of $\alpha^n$ and $\beta^m$ respectively, namely $P$ and $Q$. It only remains to show that \[ (a_1,\xi,a_2) \mbox{ and } (b_1,\eta,b_2) \]
is a linked pair, where $a_1,a_2,b_1,b_2$ are the respective projections of $c_1,c_2,d_1,d_2$. This is because $\hat\gamma$ and $\hat \delta$ meet transversally, and only once in $V$, thus $B(c_1)\cup B(c_2)$ separates $B(d_1)$ from $B(d_2)$ in $\partial V \cong S^1$. Using the orientation of $V$ induced by lifting the one of $C_{\varepsilon}\subset M$, the last fact allows us to verify the sign conditions in the definition of linked pair. It also yields that $sign(P,Q)=\epsilon(p,\gamma,\delta)$, which will be useful later. 

Note that, by the minimality of $n$ and $m$, we have $(P,Q)\in LP_2(X,Y)$. The reciprocal construction is now straightforward, and so is checking bijectivity.

% - and we can write $\hat \alpha \cap V = b_1 \omega b_2$ and $\hat \beta \cap V = c_1 \omega c_2$, where $b_1,b_2,c_1,c_2$ are pairwise different geodesic segments, and the equalities are meant in the sense of sets, without regard for the orientation.
 
 %choosing an arbitrary orientation for this arc.

%Now, for each $p$ in the intersection between $\gamma$ and $\delta$, there is $\hat{p}\in\tilde{C}_{\epsilon}$ projecting to $p$, and lifts $\hat{\gamma}$, $\hat{\delta}$ of $\gamma$ and $\delta$ meeting only at $\hat{p}$. Put $\hat \alpha =\tilde \chi \circ \hat \gamma$ and $\hat \beta = \tilde \chi \circ \hat \delta$. Since $\alpha$ and $\beta$ are reduced, the set $Y=\hat \alpha \cap \hat \beta$ is an arc $a_1 \cdots a_k$ in $\tilde C$. Let \[ V = \cup \{V(a): a\cap Y \neq \emptyset \} \] Then $V$ is a disk, with some boundary arcs. Also, we have $\hat \alpha \cap V = b_1 Y b_2$ and $\hat \beta \cap V = c_1 Y c_2$. Let $B_1 = \partial V(b_1)\cap \partial V$, the arc of $\partial V$ meeting $b_1$, and define $B_2$, $C_1$ and $C_2$ analogously. Note that $\hat \gamma$ enters $V$ through $B_1$ and leaves it through $B_2$, and also $\hat \delta$ goes from $C_1$ to $C_2$. Since $\hat p$ is their only point of intersection, $B_1$ and $B_2$ separate $C_1$ from $C_2$ in the boundary of $V$ (the ``complete'' boundary as a disk).

%Then $X=\pi(a_1)\cdots \pi(a_k)$ is the core segment of a linked pair as in section \ref{LPs} (if $Y$ is just a point, then $X$ is empty and the linked pair is of type 1). Otherwise, it would contradict lemma \ref{chas0}. The reciprocal part is similar.

\end{proof}

\begin{figure}[htbp]
\input{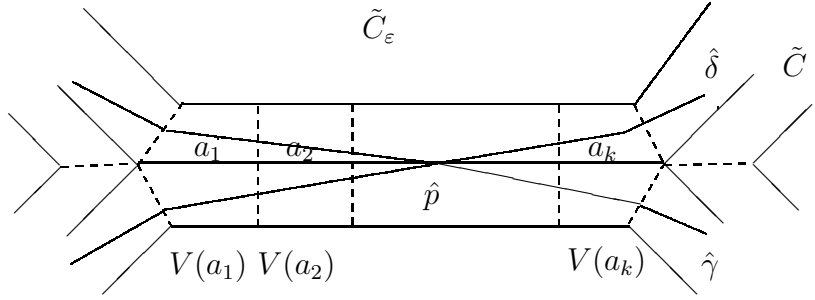}
\caption{Proof of lemma \ref{chas20}. We depict the simple case when $\hat \alpha(I)\cap\hat\beta(I)$ does not contain branching vertices of $\tilde C$ in its interior. For this figure, $a_1,\ldots,a_k$ denote the geodesic segments in $\hat \alpha(I)\cap\hat\beta(I)$.}
\end{figure}

%We  write $[\;,\;]_{\pi_1(C_{\epsilon})}$ for the Goldman-Turaev bracket in$C_{\epsilon}$, i.e. the bracket defined in \eqref{goldman}.

\begin{lemma}\label{cor:inc}
$({\mathcal S}(C), [\; ,\; ]_C)$ is a Lie algebra.
\end{lemma}

\begin{proof}
Let $X,Y\in S(C)$ and let $\alpha$, $\beta$, and $\gamma$, $\delta$ be as in Lemma \ref{chas0}. We need
to show that $\chi_*([\gamma ,
\delta]_{\pi_1(C_{\varepsilon})})=[X,Y]_C$, where $[\gamma , \delta]_{\pi_1(C_{\varepsilon})}$ stands for the Goldman-Turaev bracket between the free homotopy classes of $\gamma$ and $\delta$. This is a
consequence of Lemma \ref{chas20}: If $p$ corresponds to
the linked pair $(P,Q)$, then we have seen that $sign(P,Q)=\epsilon(p,\gamma,\delta)$ in the proof of Lemma \ref{chas20}. On the other hand, $(X\cdot_{(P, Q)}Y)$ is the conjugacy class of the image under $\chi$
of $\gamma\cdot_p\delta$, which follows from Definition \ref{dotproduct} and the properties of the correspondence between $p$ and $(P,Q)$ given by Lemma \ref{chas20}.
\end{proof}

We remark that we have obtained the isomorphism of Lie algebras \[\chi_*:({\mathcal S}_1(C_{\varepsilon}),[\;,\;]_{\pi_1(C_{\varepsilon})}) \to ({\mathcal S}(C), [\; ,\; ]_C) \] that we desired.

%Now we are in conditions to prove Theorem \ref{jacobi}.

\subsection{Proof of Theorem\ref{jacobi}}

We shall check that the bracket on $\mathcal S(M)$ given by Definition \ref{d:bracket} satisfies the axioms of a Lie algebra. It is bilinear by definition, and we would like to remark that anti-symmetry can be checked directly, showing that the bijection between $LP_2(X,Y)$ and $LP_2(Y,X)$ that sends $(P,Q)$ to $(Q,P)$ verifies that \[ (X\cdot_{(P, Q)}Y)=(Y\cdot_{(Q, P)}X) \mbox{ and } sign(P,Q)=-sign(Q,P)\] It can also be proved by the same method we use for the Jacobi identity, which we shall check next. 

Consider strings $X,Y,Z\in S(M)$, and cyclically reduced representatives $\alpha,\beta,\gamma$ of them. Applying Proposition \ref{alphabet} to $\mathcal U=\{\alpha,\beta,\gamma\}$ we see that the set $C=\alpha(I)\cup\beta(I)\cup\gamma(I)$ is a one dimensional complex with geodesic edges. Recall that the string bracket $[\;,\;]$ restricts to $\mathcal S(C)\subset \mathcal S(M)$, where it defines a Lie algebra by Lemma \ref{cor:inc}. By construction we have that $X,Y,Z$ are in $\mathcal S(C)$, thus the Jacobi identity between $X,Y,Z$ is obtained.

That shows Theorem \ref{jacobi}. We would also like to point out that there is a natural quotient $\mathcal S(M)\to\mathcal S_1(M)$, since free homotopy of closed curves is a coarser equivalence than the one defining $S(X)$, and we can show that this map is a homomorphism \[ (\mathcal S(M),[\;,\;])\to (S_1(M),[\;,\;]_{\pi_1(M)}) \]
To check this we can consider $X,Y\in S(M)$, take representatives $\alpha,\beta$ and let $\gamma,\delta$ be the curves given by Lemma \ref{chas0} for $C=\alpha(I)\cup\beta(I)$. Then $\gamma$ and $\delta$ are freely homotopic to $\alpha$ and $\beta$, and the same argument for Lemma \ref{cor:inc} shows that $[X,Y]$ maps to $[\gamma,\delta]_{\pi_1(M)}$ under the natural quotient.

As we commented in the introduction, it is also possible to define a co-bracket in a similar fashion as we did for the bracket in Definition \ref{d:bracket}, this time involving $LP_1$. That gives a Lie bi-algebra structure on $\mathcal S(M)$, and the axioms can also be verified using one dimensional complexes and results of \cite{MR2041630}.  
%In this subsection we show the Jacobi  identity for loop classes.

%Let $X$, $Y$ and $Z$ be strings in $S(M)$. We will verify the Jacobi identity for them. Let $\alpha$, $\beta$ and $\gamma$ be cyclically reduced representatives of $X$, $Y$ and $Z$, respectively, and take factorizations adapted to the three curves. Let $C=\alpha(I)\cup\beta(I)\cup\gamma(I)$. Then $C$ is a one dimensional complex and $X$, $Y$, $Z$ are in $S(C)$. Using \ref{cor:inc}, we know that $({\mathcal S}(C),[,]_C)$ is a Lie algebra. Since the inclusion ${\mathcal S}(C) \to {\mathcal S}(M)$ is clearly an injective homomorpism, $X$, $Y$ and $Z$ satisfy Jacobi identity.  The co-Jacobi identity follows the same lines using one dimensional complexes and results of  \cite{MR2041630}.

\section{Infinite lifts and intersections} \label{s.lift int}

With the goal of proving Theorem \ref{simple} in mind, we will study the intersections of a cyclically reduced curve with its inverse by looking at the horizontal lifts in the loop bundle. 

Throughout this section we fix a cyclically reduced, non-trivial closed curve $\alpha$, and write $x=\alpha(0)$. Let $\tilde \alpha$ be the horizontal lift of
$\alpha$ to ${\mathcal E}_x(M)$ such that $\tilde \alpha(0)=id_x$. We consider the set 
 
\[ \Lambda_{\alpha}=\bigcup_{n\in {\mathbb Z}} [\alpha]^n \tilde \alpha(I)  \]
which is nothing but the infinite lift of $\alpha$ through $id_x$. Since $\alpha$ is cyclically reduced, $\Lambda_{\alpha}$ is a {\em line} in ${\mathcal E}_x(M)$, i.e. is an embedding of $\R$ (it has no ``spikes''). We give it a standard orientation induced by the orientation of $\tilde{\alpha}$. Note that $\Lambda_{\alpha^{-1}}$ agrees with $\Lambda_{\alpha}$ as a set, but has the opposite orientation.

\begin{figure}[htbp]
\input{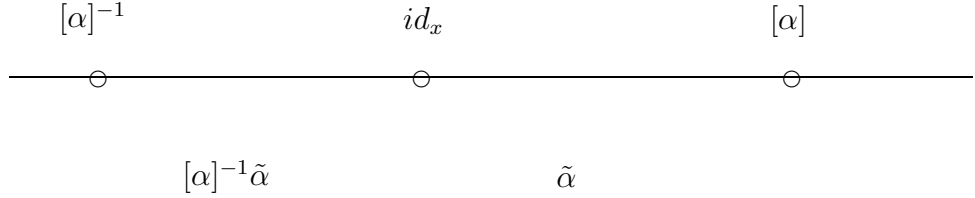}
\caption{The set $\Lambda_{\alpha}$}
\end{figure}

\subsection{Intersections as elements of $\mathcal L_x$.} \label{ss.lift int1}

For every $A\subset\mathcal E_x(M)$ define
\[ T(A)=\left\{ g \in {\mathcal L}_x(M) \;\; \mbox{such that} \;\;
g A\cap A \neq \emptyset\right\}\]
% Let $T_1(\alpha)$ be the set of
%those $g\in T(\tilde \alpha(I))$ such that
%$g\Lambda_{\alpha}\neq\Lambda_{\alpha}$. 

Consider $g\in T(\Lambda_{\alpha})$ so that $g\Lambda_{\alpha}\neq \Lambda_{\alpha}$. We show that $g\Lambda_{\alpha}\cap \Lambda_{\alpha}$ must be a compact arc (or a point): Note that $\mathcal E_x(M)$ contains no non-trivial horizontal loops, and if $g\Lambda_{\alpha}\cap \Lambda_{\alpha}$ contains a ray, then we note that $g[\alpha]^n\tilde{\alpha}(I)\subset \Lambda_{\alpha}$ for some $n$. Thus the projection of $g[\alpha]^n\tilde{\alpha}$ is some permutation of $\alpha$, which leads to an equation of the form $gw^{n+1}=w^m$ where $[\alpha]=w^k$ (for some such powers), and this implies $g\Lambda_{\alpha}= \Lambda_{\alpha}$.

Then to each $g\in T(\Lambda_{\alpha})$ with $g\Lambda_{\alpha}\neq \Lambda_{\alpha}$  we can associate a horizontal curve $b_g\subset \Lambda_{\alpha}$, with the same orientation as $\Lambda_{\alpha}$, such that $b_g(I)=g\Lambda_{\alpha}\cap \Lambda_{\alpha}$. Let $a_g \subset \Lambda_{\alpha}$ be the horizontal curve such that $ga_g(I)=b_g(I)$ and $a_g$ has the orientation carried from $b_g$ by the action of $g$ (thus we may write $ga_g=b_g$). Note that the pair $(a_g,b_g)$ determines $g$, since $\mathcal L_x(M)$ acts freely on $\mathcal E_x(M)$.

\begin{definition} \label{def:orient} Let $g\in T(\Lambda_{\alpha})$.

\begin{itemize}
\item We say that $g$ {\em preserves orientation} if either $g\Lambda_{\alpha}=\Lambda_{\alpha}$ or the orientation of $a_g$ agrees with that of $\Lambda_{\alpha}$.  Let $T^+(\Lambda_{\alpha})$ be the set of orientation preserving elements of $T(\Lambda_{\alpha})$.

%By convention, this includes the case when $a_g$ is constant.
\item  We say that $g$ {\em reverses orientation} if the orientation of $a_g$ is opposite to that of $\Lambda_{\alpha}$. We denote by $T^-(\Lambda_{\alpha})$ the set of orientation reversing elements of $T(\Lambda_{\alpha})$.
\end{itemize}

The case when $a_g$ is constant shall be regarded as both orientation preserving and reversing. When we want to exclude this case we say that $g$ {\em strictly} preserves or reverses orientation.

\end{definition}

\begin{remark} \label{inv T} The sets $T(\Lambda_{\alpha})$, $T^+(\Lambda_{\alpha})$ and $T^-(\Lambda_{\alpha})$ are closed under taking inverses.
\end{remark}

\begin{remark} \label{rm: no potencias} If $g\in T(\Lambda_{\alpha})$ and $g\Lambda_{\alpha}\neq \Lambda_{\alpha}$, then we have $l(a_g)=l(b_g)<l(\alpha)$.
\end{remark}

\begin{proof}
Assume the contrary, i.e. that $l(b_g)\geq l(\alpha)$. Then if $g\in T^+(\Lambda_{\alpha})$ we get a contradiction by showing that $g\Lambda_{\alpha}= \Lambda_{\alpha}$, with a similar argument as in the case when $g\Lambda_{\alpha}\cap \Lambda_{\alpha}$ contained a ray. On the other hand, if $g\in T^-(\Lambda_{\alpha})$ we can deduce that $\alpha^{-1}$ is a permutation of $\alpha$, which would mean that $\alpha$ is trivial. 
\end{proof}
%t_g

Next we shall see that each $g\in T(\Lambda_{\alpha})$ with $g\Lambda_{\alpha}\neq \Lambda_{\alpha}$ defines naturally an intersection pair between $\alpha$ and $\alpha^{-1}$. Let $t_g=\pi\circ a_g=\pi\circ b_g$ and $\epsilon = \pm 1$ according to whether $g$ preserves or reverses orientation (in case $a_g$ is constant the pick makes no difference). Recalling that $b_g(I)=g\Lambda_{\alpha}\cap \Lambda_{\alpha}$  and projecting what we see at a neighborhood of $b_g$ in $g\Lambda_{\alpha}\cup \Lambda_{\alpha}$, we can find geodesic curves $r,s,u,v$ such that $(r, t_g, s), (u, t^{-\epsilon}_g, v)$ is an intersection pair with $rt_gs\subset\alpha$ and $ut^{-\epsilon}_gv\subset \alpha^{-1}$, where the inclusions are modulo permutation (we do not have to consider powers of $\alpha$ or $\alpha^{-1}$ because of Remark \ref{rm: no potencias}). We depict this situation in Figure \ref{fig:2}.

\begin{figure}[htbp]
\input{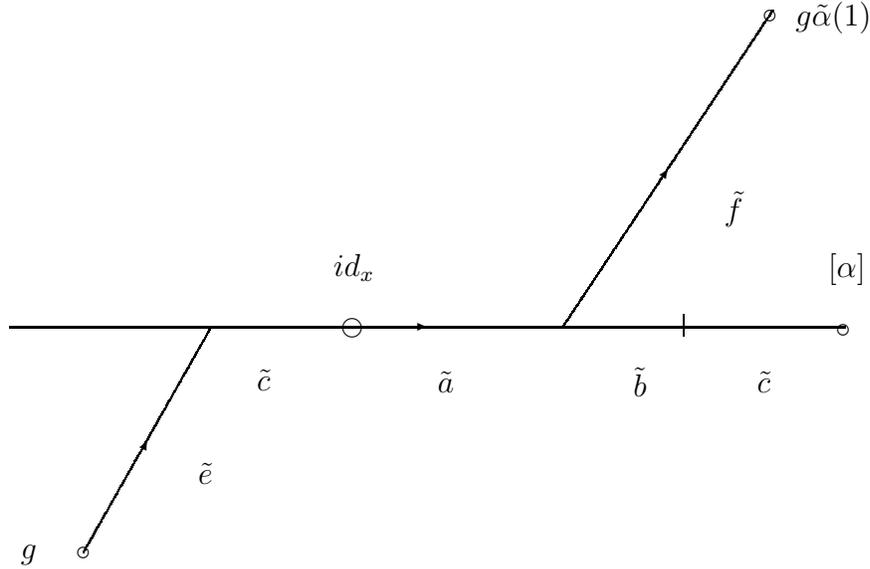}
\caption{A schematic example of $g\Lambda_{\alpha}\cup \Lambda_{\alpha}$ for $g\in T(\Lambda_{\alpha)}$. All shown curves are
horizontal. 
We assume $\alpha=abc=ecaf$ and $g=[c^{-1}e^{-1}]$, note that $t_g= ca$.} \label{fig:2}
\end{figure}

Then we can give cyclic factorizations $P_g=(\alpha_1,t_g)$ of $\alpha$ and $Q_g=(\beta_1,t_g^{-\epsilon})$ of $\alpha^{-1}$ so that:

\begin{itemize}
\item the horizontal lift of $t_g\alpha_1$ at $b_g(0)$ is contained in $\Lambda_{\alpha}$, and
\item the horizontal lift of $t_g^{-\epsilon}\beta_1$ at $a_g^{-\epsilon}(0)$ is contained in $\Lambda_{\alpha}$.
\end{itemize}

In other words, $\alpha_1$ can be obtained by projecting a curve spanning a component of  $$\Lambda_{\alpha}-\bigcup_{n\in\Z}[\alpha]^nb_g(I)$$ with the orientation given by $\Lambda_{\alpha}$, while $\beta_1$ is the analog for the translates of $a_g(I)$ and the reverse orientation to that of $\Lambda_{\alpha}$.

Notice that if $g$ is orientation preserving, i.e. when $\epsilon=1$, we get an intersection pair of type either (3) or (1) between $\alpha$ and $\alpha^{-1}$. In the orientation reversing case, when $\epsilon=-1$, we get an intersection pair of type either (2) or (1). In Figure \ref{fig:2} we depict a situation where $g$ is orientation preserving.

Observe that an element of the form $h=[\alpha]^n g [\alpha]^m$ induces the same intersection pair as $g$, since $b_h=[\alpha]^nb_g$ and $a_h=[\alpha]^{-m}a_g$ induce the same cyclic factorizations of $\alpha$ and $\alpha^{-1}$. With this in mind we define $$T_1(\alpha) =\{g\in T(\tilde{\alpha}(I)) \,:\, g\Lambda_{\alpha}\neq \Lambda_{\alpha} \mbox{ and }\tilde{\alpha}(1)\notin a_g(I)\cup b_g(I)  \} $$ 
noting that the conditions amount to ask that $a_g$ and $b_g$ meet $\tilde{\alpha}(I)$ but not $\tilde{\alpha}(1)$. Since $\tilde{\alpha}([0,1))$ is a fundamental domain for $\Lambda_{\alpha}$ under translations by powers of $[\alpha]$, we get the following: 

\begin{remark}\label{mn}
Let $g\in T(\Lambda_{\alpha})$ such that  $g\Lambda_{\alpha}\neq \Lambda_{\alpha} $. Then there is a unique $g_1\in T_1(\alpha)$ and integers $m,n$ such that 
\[g=[\alpha]^n g_1[\alpha]^m.    \]

\end{remark}

The next result describes the intersections of a string with its inverse in terms of elements of the loop group. Let $\mathcal I(\alpha, \alpha^{-1})$ denote the set of intersection pairs between  $\alpha$ and $\alpha^{-1}$. 

\begin{lemma}\label{bijection}
 The map \[ T_1(\alpha)\to \mathcal I(\alpha,\alpha^{-1})\] that takes $g\to (P_g, Q_g)$ is a bijection.
\end{lemma}
\begin{proof}
For injectivity, consider the way in which an horizontal curve $\nu\subset\Lambda_{\alpha}$ with $l(\nu)<l(\alpha)$ defines a cyclic factorization of $\alpha$ or $\alpha^{-1}$, as was used in the construction of the map $g\to (P_g, Q_g)$. Then observe that two such curves $\nu_1$ and $\nu_2$ yield the same cyclic factorization iff $\nu_1=[\alpha]^n\nu_2$ for some $n\in \Z$. Combining this fact with Remark \ref{mn} gives injectivity.

To show surjectivity consider $(P,Q)\in \mathcal I(\alpha,\alpha^{-1})$ and write $P=(\alpha_1,\xi)$, $Q=(\beta_1,\eta)$ as in Definition \ref{def:lp cyclic}. Let $\epsilon=\pm 1$ so that $\eta=\xi^{-\epsilon}$, i.e. $\epsilon=1$ for $(P,Q)$ of type (3), while $\epsilon=-1$ for type (2), and either one for type (1). Since $\xi\alpha_1$ is a permutation of $\alpha$, there are horizontal curves $\tilde\xi$ and $\tilde\alpha_1$, projecting to $\xi$ and $\alpha_1$ respectively, such that the concatenation $\tilde\xi\tilde\alpha_1$ is well defined and contained in $\Lambda_{\alpha}$. Note that $\tilde\xi$ and $\tilde\alpha_1$ have the orientation of $\Lambda_{\alpha}$, and that $\Lambda_{\alpha}$ is the infinite lift of $\xi\alpha_1$ that continues $\tilde\xi\tilde\alpha_1$.   Now we consider the infinite horizontal lift of $\eta\beta_1$ starting at $\tilde\xi^{-\epsilon}(0)$. Since $\eta\beta_1$ is a permutation of $\alpha^{-1}$, this infinite lift must be of the form $g\Lambda_{\alpha}$ for some $g\in T(\Lambda_{\alpha})$. Notice that the definition of intersection pair implies, by taking the appropriate horizontal lifts, that $\Lambda_{\alpha}\cap g\Lambda_{\alpha} = \tilde\xi(I)$. Thus it becomes direct to verify that $(P_g,Q_g)=(P,Q)$, and we can use Remark \ref{mn} to finish the proof.

\end{proof}

Through the proof of Lemma \ref{bijection} we see that $T_1(\alpha)$ is a choice of a restriction of domain, in order to obtain a bijection from the construction that associates $g\to (P_g, Q_g)$. This choice satisfies the following nice property:
 
\begin{remark}\label{inverse}
$g\in T_1(\alpha)$ iff $g^{-1}\in T_1(\alpha)$. Moreover, $$b_{g^{-1}}=a_g^{\epsilon} \mbox{ and } a_{g^{-1}} = b_g^{\epsilon}$$ where $\epsilon=\pm 1$ according to whether $g$ is orientation preserving or reversing.

\end{remark}

In later sections we shall focus on the linked pairs, i.e. the intersection pairs that are relevant for the bracket.

\begin{definition} \label{def:T0} We define $T_0(\alpha)\subseteq T_1(\alpha)$ as the set of elements that correspond to linked pairs under the bijection of Lemma \ref{bijection}.
\end{definition}

Note that by Remark \ref{rm: no potencias} and Lemma \ref{bijection}, the set $T_0(\alpha)$ is in bijection with $LP_2(X,X^{-1})$, where $X$ is the conjugacy class of $[\alpha]$. From these same results we also get that $LP_1(X)=LP_2(X,X)$, which is in natural bijection with $LP_2(X,X^{-1})$. 
Observe also that by Lemma \ref{chas20} a string $X$ is simple, as defined in the introduction, iff $LP_1(X)=\emptyset$, or equivalently, iff $LP_2(X,X^{-1})=\emptyset$.

\subsection{Orientation reversing elements and unique intersections.}
 
The orientation properties of the elements of $T_0(\alpha)$, which correspond to the type of their associated linked pairs, will play a major role in proving Theorem \ref{simple}. Next we study the key properties of the orientation reversing case.

\begin{lemma} \label{reversing0} Let $\xi$ and $\eta$ be non-constant segments of $\Lambda_{\alpha}$ going in the positive orientation. Then if $\pi\circ\xi=(\pi\circ\eta)^{-1}$  we have $\xi(I)\cap\eta(I)=\emptyset$ and $l(\alpha)>2l(\xi)$. 

\end{lemma}

\begin{proof}
Write $\gamma=\pi\circ\xi$. Then an overlap between  $\xi$ and $\eta$ would project to an overlap between $\gamma$ and $\gamma^{-1}$, and if $\xi$ and $\eta$ meet at an endpoint, that would project to a reduction of $\alpha$, of the form $\gamma\gamma^{-1}$ or $\gamma^{-1}\gamma$. Thus we get the first claim. The second one comes from considering a permutation $\alpha_0$ of $\alpha$ so that its horizontal lift starting at $\xi(0)$ is contained in $\Lambda_{\alpha}$. Note that such lift ends at $[\alpha]\xi(0)$, which is not in $\eta(I)$ by the first claim applied to $\eta$ and $[\alpha]\xi$. Thus we obtain $\alpha_0=\gamma a \gamma^{-1} b$ with $a$ and $b$ non-constant, so $$l(\alpha)=l(\alpha_0)>2l(\gamma)=2l(\xi).$$

\end{proof}

\begin{lemma}\label{reversing}
Let $g\in T^-(\Lambda_{\alpha})$. Then $[\alpha^n]a_g(I)\cap [\alpha^m]b_g(I)=\emptyset$ for all $n,m\in\Z$.
\end{lemma}

\begin{proof}
If $g$ reverses orientation strictly, we apply lemma \ref{reversing0} to the curves $[\alpha^n]a_g^{-1}$ and $[\alpha^m]b_g$. 
In case $a_g$ is constant, say $a_g=e_p$ for $p\in\Lambda_{\alpha}$, we get that $[\alpha^n]g[\alpha^{-m}]p=p$ which is absurd because the action of $\mathcal L_x(M)$ is free and $g$, being orientation reversing, is not a power of $\alpha$.
\end{proof}

\begin{remark} Note that for $g\in T^-(\Lambda_{\alpha})$ we have $l(t_g)<l(\alpha)/2$. 
\end{remark}

 We say that $\alpha$ has {\em unique intersection} if $T_1(\alpha)$ consists of only two elements, $g$ and $g^{-1}$ (by Remark \ref{inverse}).   By Lemma \ref{bijection} this is equivalent to say that $\mathcal I(\alpha,\alpha^{-1})$ has two elements. Observe that the definition of intersection pair makes sense for a curve in a general one dimensional complex, i.e. not necessarily embedded in a surface. Thus we may speak of unique intersection for curves in this more general setting.

Given a subgroup $G\subseteq \mathcal L_x(M)$  we can consider $\Lambda_{\alpha}/G$, the image of $\Lambda_{\alpha}$ in the quotient $\mathcal E_x(M)/G$, which is a one dimensional complex since $\Lambda_{\alpha}/\mathcal L_x(M)=\alpha(I)$ and $\alpha$ is piecewise geodesic. For $g\in \mathcal L_x(M)$, let $G_g$ be the subgroup generated by $g$ and $[\alpha]$.

\begin{lemma}\label{quotient}
Let $g\in T^-(\Lambda_{\alpha})$ and $\bar \alpha$ be the projection of $\tilde \alpha$ onto $\Lambda_{\alpha}/G_g$.
Then $\bar \alpha$ has a unique intersection.
\end{lemma}

\begin{proof}
By Lemma \ref{reversing} we have that all the translates $[\alpha]^na_g(I)$ and $[\alpha]^mb_g(I)$, for $n,m\in\Z$, are pairwise disjoint. Consider first the quotient $\Lambda_{\alpha}/\langle [\alpha]\rangle$, which is a circle that can be parametrized by the projection of $\tilde \alpha$. Then $a_g(I)$ and $b_g(I)$ project to  $\Lambda_{\alpha}/\langle [\alpha]\rangle$ as two disjoint intervals, and $\Lambda_{\alpha}/G_g$ is the further quotient obtained by identifying these two intervals (with the appropriate orientation). The interval resulting from this identification is the core of the only self-intersection pairs of $\bar\alpha$, which are only two (of the form $(P,Q)$ and $(Q,P)$), and this  shows the lemma. 

%By Lemma \ref{reversing} we have $G_g a_g(I)\cap \tilde \alpha(I)=b_g(I)$, i.e. the only identifications are $a_g(I)\sim b_g(I)$ and $\tilde\alpha(0)\sim \tilde\alpha(1)$.  
\end{proof}

\subsection{$\alpha$-oriented subgroups}\label{alpha_oriented} 

Now we turn our attention to the orientation preserving elements of $T(\Lambda_{\alpha})$. We will be showing that they generate subgroups with the following property.

\begin{definition}
Let $G\subset \mathcal L_x(M)$ be a subgroup. We  say that $G$ is {\em $\alpha$-oriented} if for every $g, h\in G$ with $g^{-1}h\in T(\Lambda_{\alpha})$ we have
$g^{-1}h\in T^+(\Lambda_{\alpha})$.
\end{definition}

Let us explain this definition in geometric terms. First note that each $g\in \mathcal L_x(M)$ induces an orientation on $g\Lambda_{\alpha}$ by carrying the standard orientation of $\Lambda_{\alpha}$ through the action of $g$. For two elements $g,h\in \mathcal L_x(M)$, the condition that $g^{-1}h\in T(\Lambda_{\alpha})$ is equivalent to saying that $g\Lambda_{\alpha}\cap h\Lambda_{\alpha}\neq\emptyset$, and we have $g^{-1}h\in T^+(\Lambda_{\alpha})$ exactly when the orientations of $g\Lambda_{\alpha}$ and $h\Lambda_{\alpha}$ agree on their (non-empty) intersection. If $G\subset \mathcal L_x(M)$ is a subgroup we have that $$ G\Lambda_{\alpha} = \bigcup_{g\in G}g\Lambda_{\alpha}$$ is a one dimensional complex whose connected components are simplicial trees, recalling the form of the intersections between translates of $\Lambda_{\alpha}$ by elements of $\mathcal L_x(M)$. Then we have: 

\begin{remark}
If $G$ is $\alpha$-oriented we can give $G\Lambda_{\alpha}$ a $G$-invariant orientation that extends the standard orientation of $\Lambda_{\alpha}$. 
\end{remark}

If a subgroup $G\subset \mathcal L_x(M)$ is generated by some elements of $T(\Lambda_{\alpha})$ then $G\Lambda_{\alpha}$ is connected, thus is a simplicial tree. The following technical result will be useful in this context.  

\begin{lemma} \label{convexity} Suppose $g,h_1,\ldots,h_n\in \mathcal L_x(M) $ satisfy that:
\begin{itemize}
\item $g^{-1}h_1, g^{-1}h_n \in T(\Lambda_{\alpha})$,
\item $g^{-1}h_i\notin T(\Lambda_{\alpha})$ for $i=2,\ldots,n-1$,
\item $h_{i-1}^{-1}h_i \in T(\Lambda_{\alpha})$ for $i=2,\ldots,n$.
\end{itemize}

Then $h_1^{-1}h_n \in T(\Lambda_{\alpha})$.
\end{lemma}

\begin{proof}
Interpreting the hypotheses in terms of intersections of translates of $\Lambda_{\alpha}$, we can find a curve $\beta\subset h_1\Lambda_{\alpha}\cup\cdots\cup h_n\Lambda_{\alpha}$ that only meets $g\Lambda_{\alpha}$ at its endpoints, with $\beta(0)\in h_1\Lambda_{\alpha}\cap g\Lambda_{\alpha}$ and $\beta(1)\in h_n\Lambda_{\alpha}\cap g\Lambda_{\alpha}$. Since $g\Lambda_{\alpha}\cup h_1\Lambda_{\alpha}\cup\cdots\cup h_n\Lambda_{\alpha}$ has no non-trivial loops, we must have $\beta(0)=\beta(1)$ (and $\beta$ must be a trivial loop), which provides a point in $h_1\Lambda_{\alpha}\cap h_n\Lambda_{\alpha}$ as desired.
\end{proof}

%\begin{remark} Note that if $G$ is $\alpha$-oriented we can orient teh union of  $g\Lambda_{\alpha}$, $g\in G$,  according to the orientation of $\Lambda_{\alpha}$ so that $g$ preserves orientation.  If a lift is contained in one $g\Lambda_{\alpha}$ we say that is positively oriented if it has the orientation of $g\Lambda_{\alpha}$.  \end{remark}

The following is a straightforward observation. 

\begin{lemma}\label{fundamental} 
If $g, h\in  T^+(\Lambda_{\alpha})$ and  $g^{-1}h\in  T(\Lambda_{\alpha})$, then  $g^{-1}h\in T^+(\Lambda_{\alpha})$.

\end{lemma}

Next we present the main result of this subsection, concerning the subgroups generated by orientation preserving elements.

\begin{lemma}\label{oriented}
Let $G\subset\mathcal L_x(M)$ be a finitely generated subgroup whose generators belong to  $T^+(\Lambda_{\alpha})$. Then $G$ is $\alpha$-oriented. 
\end{lemma}

\begin{proof}

By hypothesis we can write
\[G=\bigcup_i G_i, \mbox{ where } G_0=\{id_x\}, \mbox{ and } G_{i+1}=G_i\cup\{ h_{i+1}\}, \]
such that there is $h'_i\in G_i$ with $h_{i+1}^{-1}h'_i\in  T^+(\Lambda_{\alpha})$ for every $i\geq 0$. We are going to show the lemma by induction on $i$: assuming that $g^{-1}h\in T(\Lambda_{\alpha})$ implies $g^{-1}h\in T^+(\Lambda_{\alpha})$ for $g, h\in G_i$, we shall show that this same property holds for $g,h\in G_{i+1}$. The base case of this induction is trivial.  

By Remark \ref{inv T}, we only need to consider the case when $g=h_{i+1}$ and $h\in G_i$. So let $h\in G_i$ be such that $h^{-1}_{i+1}h\in T(\Lambda_{\alpha})$, and we are going  to show that $h^{-1}_{i+1}h\in T^+(\Lambda_{\alpha})$. By construction of the set $G_i$  there is a sequence $h'_i=k_1,\ldots,k_n=h$ in $G_i$ such that $k_j^{-1}k_{j+1}\in T^+(\Lambda_{\alpha})$ for $j=1,\ldots,n-1$. Since $k_j\in G_i$ for every $j$, the induction hypothesis gives us that $k_j^{-1}k_l\in T^+(\Lambda_{\alpha})$ whenever 
$k_j^{-1}k_l\in T(\Lambda_{\alpha})$. 

By lemma \ref{convexity} we can assume, maybe after taking a subsequence,  that 
$h^{-1}_{i+1}k_j\in T(\Lambda_{\alpha})$ for every $j$.  Now write
\[ h^{-1}_{i+1}k_2= (h^{-1}_{i+1}k_1) (k^{-1}_{1}k_2)  \]
and note that  $h^{-1}_{i+1}k_1$   and    $k^{-1}_{1}k_2$ are  orientation preserving by construction. Therefore $h^{-1}_{i+1}k_2$ is orientation preserving by 
Lemma \ref{fundamental}.  Proceeding inductively we conclude that $h^{-1}_{i+1}k_n$ is also orientation preserving, as desired.

\end{proof}

\section{Formulas for the terms of the bracket}

In this section we study the dot products between a string and its inverse  applying what we developed in section \ref{s.lift int}. Again we fix a cyclically reduced, non-trivial closed curve $\alpha$, and let $x=\alpha(0)$ and $X\in S(M)$ be the conjugacy class of $\alpha$. 

\subsection{Expressions for the dot product}

Let $g\in T_1(\alpha)$ and recall the horizontal curves $(a_g,b_g)$ defined in subsection \ref{ss.lift int1}. We introduce the following curves:

\begin{itemize}
\item $\tilde\alpha_g$ is the segment of $\Lambda_{\alpha}$ starting at $b_g(1)$ and ending at $[\alpha]b_g(1)$. Let $\alpha_g=\pi\circ\tilde\alpha_g$.

\item $\tilde\beta_g$ is the segment of $g\Lambda_{\alpha}$ starting at $b_g(1)$ and ending at $g[\alpha^{-1}]a_g(1)$. Let $\beta_g=\pi\circ\tilde\beta_g$.

\item $\tilde\gamma_g$ is the segment of $\Lambda_{\alpha}$ starting at $id_x$ and ending at $b_g(1)$. Let $\gamma_g=\pi\circ\tilde\gamma_g$.
  
\end{itemize}

We observe $\alpha_g$ and $\beta_g$ are permutations of $\alpha$ and $\alpha^{-1}$ respectively, and that $\tilde\alpha_g$ and $\tilde\beta_g$ are their respective horizontal lifts starting at $b_g(1)$. This is easy to see for $\alpha_g$, and in the case of $\beta_g$ note that $\tilde\beta_g=g\tilde\beta'$ where $\tilde\beta'$ is the segment of $\Lambda_{\alpha}$ starting at $a_g(1)$ and ending at $[\alpha^{-1}]a_g(1)$. According to Definition \ref{dotproduct} (which also makes sense for intersection pairs), the conjugacy class of $\alpha_g\beta_g$ is the dot product $(X\cdot_{(P,Q)}X^{-1})$ where $(P,Q)$ is the intersection pair corresponding to $g$.

On the other hand, $\gamma_g$ is the curve that gives the change of basepoint conjugation so that $[\gamma_g\alpha_g\gamma_g^{-1}]$ and $[\gamma_g\beta_g\gamma_g^{-1}]$ belong to $\mathcal L_x(M)$. We also have that $\tilde\gamma_g$ is the horizontal lift of $\gamma$ at $id_x$. For an example of these curves, in figure \ref{fig:2} we have $\alpha_g=bca$, $\beta_g=a^{-1}c^{-1}e^{-1}f^{-1}$ and $\gamma_g=a$.

\begin{remark} \label{rm:inv a b} By Remark \ref{inverse} we have

\begin{itemize}
\item If $g$ preserves orientation, then $\alpha_{g^{-1}} = \beta_g^{-1} \mbox{ and } \beta_{g^{-1}} = \alpha_g^{-1}$. 

\item If $g$ reverses orientation, $\alpha_{g^{-1}}$ and $\beta_{g^{-1}}$ are the respective reductions of $t_g\beta_g^{-1}t_g^{-1}$ and $t_g\alpha_g^{-1}t_g^{-1}$. 

\end{itemize}
(Note that in the first case $t_{g^{-1}}=t_g$, while in the second case $t_{g^{-1}}=t_g^{-1}$).

\end{remark}

This gives a relationship between the dot products between $X$ and $X^{-1}$ associated to $g$ and $g^{-1}$. 

\begin{remark} \label{dot product of inverses} Let $(P,Q)\in \mathcal I(\alpha,\alpha^{-1})$ correspond to $g\in T_1(\alpha)$, and denote by $(Q',P')$ the intersection pair corresponding to $g^{-1}$. Then we have $$(X\cdot_{(Q',P')}X^{-1})=(X\cdot_{(P,Q)}X^{-1})^{-1} $$
\end{remark}

In particular, these dot products cannot be equal, as a non-trivial loop is not conjugate to its inverse. Next we define the curves that will help us write reduced forms for the dot products.

\begin{definition} \label{d:c1c2} Let $g\in T_1(\alpha)$. We define the curves $c_1(\alpha,g)$ and $c_2(\alpha,g)$ according to whether $g$ preserves of reverses orientation:

\begin{itemize}
\item If $g$ preserves orientation, $$c_1(\alpha,g) \mbox{ is the reduced form of } \alpha_gt_g^{-1}  \mbox{ and }$$ $$ c_2(\alpha,g) \mbox{ is the reduced form of } t_g\beta_g$$

\item If $g$ reverses orientation, let $$c_1(\alpha,g)= \alpha_g \mbox{ and } c_2(\alpha,g)= \beta_g $$

\end{itemize}

\end{definition}

As an example, in Figure \ref{fig:2} we have $c_1(\alpha, g)=b$ and
$c_2(\alpha, g)=e^{-1}f^{-1}$. 

Let us interpret this definition in terms of the construction of the intersection pair $(P_g,Q_g)$ given in subsection \ref{ss.lift int1}. In the orientation preserving case we had $P_g=(\alpha_1,t_g)$ and $Q_g=(\beta_1,t_g^{-1})$, and recalling the construction we get  $$\alpha_g=\alpha_1 t_g \mbox{ and } \beta_g = t_g^{-1}\beta_1, \mbox{ thus } c_1(\alpha,g)=\alpha_1 \mbox{ and } c_2(\alpha,g)=\beta_1$$
In the orientation reversing case we had $P_g=(\alpha_1,t_g)$ and $Q_g=(\beta_1,t_g)$ and we get $$\alpha_g=\alpha_1 t_g \mbox{ and } \beta_g = \beta_1 t_g, \mbox{ and so } c_1(\alpha,g)=\alpha_1 t_g \mbox{ and } c_2(\alpha,g)=\beta_1 t_g$$

In both cases we have that the concatenation $c_1(\alpha,g)c_2(\alpha,g)$ is the cyclically reduced form of $\alpha_g\beta_g$. Also note that:

\begin{remark}\label{pair2}
If $g\in T_1(\alpha)$ is orientation reversing, then $c_1(\alpha,g)$ is  a permutation of $\alpha$ and $c_2(\alpha,g)$ is a permutation of $\alpha^{-1}$. In particular $l(c_1(\alpha,g))=l(c_2(\alpha,g))=l(\alpha)$.
\end{remark}

\begin{remark}\label{pair3}
On the other hand, if $g\in T_1(\alpha)$ is orientation preserving we have $l(c_1(\alpha,g))=l(c_2(\alpha,g))=l(\alpha)-l(t_g)$. 
\end{remark}

So in any case the lengths of $c_1(\alpha,g)$ and $c_2(\alpha,g)$ agree.
We can also deduce the length of the dot product (i.e. the length of a cyclically reduced form), as follows:

\begin{remark} \label{r:int length}

For $g\in T_1(\alpha)$ we have: 

\begin{itemize}
\item ${l(c_1(\alpha,g)c_2(\alpha,g))= 2l(\alpha)-2l(t_g)}$ if $g$ preserves orientation, and

\item $l(c_1(\alpha,g)c_2(\alpha,g))= 2l(\alpha)$ if $g$ reverses
orientation.

\end{itemize} 

\end{remark}

Recalling the relationship between orientation and the type of intersection pairs from subsection \ref{ss.lift int1}, Remark \ref{r:int length} implies that if  $c_1(\alpha,g)c_2(\alpha,g)$ is a permutation of $c_1(\alpha,h)c_2(\alpha,h)$ for $g,h\in T_1(\alpha)$ then either:

\begin{itemize}
\item $(P_g,Q_g)$ and $(P_h,Q_h)$ are both of types (2) or (1) (i.e. $g$ and $h$ are orientation reversing), or
\item $(P_g,Q_g)$ and $(P_h,Q_h)$ are both of type (3) (i.e. $g$ and $h$ are strictly orientation preserving), and $l(t_g)=l(t_h)$.
\end{itemize}

This observation is an example of recovering information about the intersection pair from the corresponding dot product. In the following sections we will be proving stronger results within this same idea, which will ultimately lead us to Theorem \ref{simple} by showing there can be no cancellations in the formula for $[X,X^{-1}]$. 

Next we record the behaviour of the curves from Definition \ref{d:c1c2} under taking inverses in $T_1(\alpha)$, which we can compute from Remark \ref{rm:inv a b}.

\begin{remark} \label{rm:inv c1c2} Let $g\in T_1(\alpha)$, then
\begin{itemize}
\item if $g$ preserves orientation, $$c_1(\alpha,g^{-1})=c_2(\alpha,g)^{-1}  \mbox{ and } c_2(\alpha,g^{-1})=  c_1(\alpha,g)^{-1}$$

\item if $g$ reverses orientation, 
$$c_1(\alpha,g^{-1}) \mbox{ is the reduced form of } t_g  c_2(\alpha,g)^{-1} t_g^{-1} \mbox{ and }$$ $$ c_2(\alpha,g^{-1}) \mbox{ is the reduced form of } t_g  c_1(\alpha,g)^{-1} t_g^{-1}$$
\end{itemize}

\end{remark}

The next result lets us write the dot products as commutators in $\mathcal L_x(M)$.

\begin{lemma}\label{formula2}
 Let $g\in T_1(\alpha)$, and put $c_i=c_i(\alpha,g)$ for $i=1,2$. Then we have
\begin{enumerate}
\item
$[\alpha]g[\alpha]^{-1}g^{-1}= [\gamma_g c_1 c_2 \gamma_g^{-1}] $, and 
\item 
if $(P,Q)\in \mathcal I(\alpha,\alpha^{-1})$ corresponds to $g$ by the bijection of Lemma \ref{bijection}, then 
$(X \cdot_{(P,Q)}X^{-1})$ is the conjugacy class of $[\alpha]g[\alpha]^{-1}g^{-1}$.
\end{enumerate}
\end{lemma}

\begin{proof} The second point follows from the first and the fact that the cyclically reduced form of $[\gamma_g c_1 c_2 \gamma_g^{-1}]$, which is $c_1c_2$, represents the dot product $(X \cdot_{(P,Q)}X^{-1})$ as discussed previously in this subsection.

We show the first point in the statement for $g$ corresponding to an intersection pair of type (3), as the other cases result from a straightforward adaptation of the same computations. In fact, type (1) can be considered within types either (2) or (3) by allowing constant core curves.

Whithin the case of $(P,Q)$ being of type (3), we distinguish 3 subcases according to whether $id_x$ belongs to both $a_g$ and $b_g$, to only one of them, or to neither of them. First we assume that $id_x$ is in $b_g$ but not in $a_g$. This is the situation shown in Figure \ref{fig:2}. Then we can write $\alpha=abc=ecaf$ with $g=[c^{-1}e^{-1}]$, noting that it corresponds to an intersection pair of type (3). We have $\gamma_g=a$, $c_1=b$ and $c_2=e^{-1}f^{-1}$. Thus we get $[\gamma_g c_1 c_2 \gamma_g^{-1}]=abe^{-1}f^{-1}a^{-1}$. On the other hand we compute
\begin{equation}\label{eq1}
g[\alpha]^{-1}g^{-1}=[(c^{-1}e^{-1}) (f^{-1}a^{-1}c^{-1}e^{-1})(ec)]
\end{equation}
Thus
\[[\alpha]g[\alpha]^{-1}g^{-1}=[(abc)(c^{-1}e^{-1}) (f^{-1}a^{-1}c^{-1}e^{-1})(ec)  ]= [\gamma_gc_1c_2\gamma_g^{-1}] \]
as desired. The situation is symmetrical for $id_x$ in $a_g$ but not in $b_g$.

In the second subcase, when $id_x$ is in both $a_g$ and $b_g$, we have that $g\in b_g(I)$. If $g$ lies before $id_x$ in the orientation of $\Lambda_{\alpha}$ we can write $\alpha= abcc' = c'afc$, where $t_g=cc'a$ and $g=[c'^{-1}]$. Now $\gamma_g=a$, $c_1=b$ and $c_2=f^{-1}$, and we compute \begin{equation}\label{eq2}
g[\alpha]^{-1}g^{-1}=[c'^{-1}(c^{-1}f^{-1}a^{-1}c'^{-1})c'] 
\end{equation}
and 
\[ [\alpha]g[\alpha]^{-1}g^{-1} = [(abcc')c'^{-1}(c^{-1}f^{-1}a^{-1}c'^{-1})c'] = [abf^{-1}a^{-1}],  \]
that is $[\gamma_g c_1c_2\gamma_g^{-1}]$.  When $g$ lies after $id_x$ in $\Lambda_{\alpha}$ we write $\alpha=aa'bc = a'fca$, where $t_g=caa'$ and $g=[a']$, and the computation is similar.

Finally, if $id_x$ is neither in $a_g$ nor $b_g$, we have $\alpha=abc=ebf$ with $t_g=b$ and $g=[ae^{-1}]$. We see that $\gamma_g=ab$, $c_1=ca$ and $c_2=e^{-1}f^{-1}$. On the other hand
\begin{equation}\label{eq3}
g[\alpha]^{-1}g^{-1}=[(ae^{-1}) (f^{-1}b^{-1}e^{-1})(ea^{-1})]
\end{equation}
thus we get
\[ [\alpha]g[\alpha]^{-1}g^{-1}= [ab(ca e^{-1} f^{-1})b^{-1} a^{-1} ]=[\gamma_g c_1c_2\gamma_g^{-1}]  \] 
\end{proof}

\subsection{Conjugate dot products}

Our strategy for Theorem \ref{simple} is to show that for a primitive string $X$ there can be no cancellations in the formula for $[X,X^{-1}]$ given in Definition \ref{d:bracket}. This would imply that $[X,X^{-1}]=0$ only when $LP_2(X,X^{-1})=\emptyset$, provided that $X$ is primitive, thus proving Theorem \ref{simple}. So we need to study what happens if two linked pairs between $X$ and $X^{-1}$ yield the same dot product. Here we shall focus on what we can achieve for general intersection pairs, leaving the discussion of linked pairs and their signs for the next section.

%allows us to recover some information about the linked pair $(P,Q)$ from the dot product $(X\cdot_{(P,Q)}X^{-1})$. In the following sections we will be proving stronger results on this same idea,  ultimately allowing us to show that there are no cancellations in the formula for $[X,X^{-1}]$, thus showing Theorem \ref{simple}.   

Assume that $g, h\in T_1(\alpha)$ are such that $c_1(\alpha,g)c_2(\alpha, g)$ is a permutation (maybe trivial) of the curve $c_1(\alpha,
h)c_2(\alpha,h)$, which is to say that their corresponding intersection pairs yield the same dot product. So we have
\[[c_1(\alpha, h)c_2(\alpha, h)]=[rc_1(\alpha,g)c_2(\alpha, g) r^{-1}] \]
where $r\in\Omega$ is an initial segment of $c_1(\alpha, h)c_2(\alpha, h)$.

 By Remark \ref{r:int length} and the discussion preceeding it, the curves $c_1(\alpha, h)$, $c_2(\alpha, h)$, $c_1(\alpha, g)$ and $c_2(\alpha, g)$ have all the same length. Thus we may assume that $r$ is an initial segment of $c_1(\alpha, h)$ (otherwise we exchange the roles of $g$ and $h$), and find $s,\,t,\,u\in\Omega$ such that
\begin{equation}\label{c1h}
c_1(\alpha, h)=rs \;\;\;\;\; c_2(\alpha, h)=tu
\end{equation}
\begin{equation}\label{c1g}
c_1(\alpha, g)=st \;\;\;\;\; c_2(\alpha, g)=ur
\end{equation}
where  $l(r)=l(t)$ and $l(s)=l(u)$. In particular, $r$ is constant iff $t$ is constant (trivial permutation case), and $s$ is constant iff $u$ is constant. (Note: $t$ is not to be confused with $t_g$ nor $t_h$). Let
\begin{equation}\label{8}
\phi=[\gamma_h r \gamma_g^{-1}]\in\mathcal L_x(M) 
\end{equation}
Then by Lemma \ref{formula2} we have 
\begin{equation} \label{phi conj}
\phi[\alpha]g [\alpha]^{-1}g^{-1}\phi^{-1}= [\alpha]h[\alpha]^{-1}h^{-1}
\end{equation} 
Lemma \ref{formula2} also gives the converse: if $g, h\in T_1(\alpha)$ are so that $[\alpha]g [\alpha]^{-1}g^{-1}$ and $[\alpha]h [\alpha]^{-1}h^{-1}$ are conjugate in $\mathcal L_x(M)$, then $c_1(\alpha,g)c_2(\alpha, g)$ is a permutation of $c_1(\alpha,h)c_2(\alpha,h)$.

\begin{lemma}\label{reversing2} Assume that $\alpha$ is primitive, and that $g, h\in T_1(\alpha)$ are orientation reversing and so that 
$[\alpha]h[\alpha]^{-1}h^{-1}$ and $[\alpha]g[\alpha]^{-1}g^{-1}$
are conjugate. Then 
\[g=h \] in particular
\[ c_i(\alpha,g)= c_i(\alpha,h) \qquad \mbox{ for }i=1,2.\]

\end{lemma}

\begin{proof}
Recall that since $g$ and $h$ are orientation reversing we have $c_1(\alpha,g)=\alpha_g$, $c_2(\alpha,g)=\beta_g$, $c_1(\alpha,h)=\alpha_h$ and $c_2(\alpha,h)=\beta_h$, which will be useful through the proof.

Let $\bar \alpha$ be the projection of $\tilde \alpha$ onto $\Lambda_{\alpha}/G_h$  as in Lemma \ref{quotient}. Note that $\mathcal E_x(M)/G_h$ is an intermediate bundle over $M$, and has a notion of horizontal lift, by projecting the one in $\mathcal E_x(M)$  (which is equivariant). Throughout this proof we shall consider horizontal lifts to $\mathcal E_x(M)/G_h$ repeatedly, and refer to them simply as ``lifts''.

Let $\bar \gamma_h$ be the lift of $\gamma_h$ starting at $\bar\gamma_h(0)=\bar \alpha(0)$, which is the same as the projection of $\tilde\gamma_h$, thus it is contained in $\Lambda_{\alpha}/G_h$ and $\bar\gamma_h(1)$ is the projection of $b_h(1)$. Let $\bar c_1(\alpha,h)\bar c_2(\alpha,h)$ be the lift of $c_1(\alpha,h) c_2(\alpha,h)$ beginning at $\bar \gamma_h(1)$. We see that this curve is closed and contained in $\Lambda_{\alpha}/G_h$, by observing that $\bar c_1(\alpha,h)$ and $\bar c_2(\alpha,h)$ are the respective projections of $\tilde\alpha_h\subset \Lambda_{\alpha}$ and $h^{-1}\tilde\beta_h\subset \Lambda_{\alpha}$, thus each one is a closed curve at $\bar\gamma_h(1)$. 

Next we take $\bar r$ the lift of $r$ starting at $\bar\gamma_h(1)$, noting that it is an initial segment of $\bar c_1(\alpha,h)$. Consider $\bar c_1(\alpha,g)\bar c_2(\alpha,g)$ the lift of $c_1(\alpha,g)c_2(\alpha,g)$ that starts at $\bar r(1)$ (recalling that $c_1(\alpha,g)$ begins at $r(1)$).

{\sl Claim 1}: $\bar c_1(\alpha,g)\bar c_2(\alpha,g)$ is closed and contained in $\Lambda_{\alpha}/G_h$.

To show this claim, write $c_1(\alpha,g)c_2(\alpha,g) = stur$, and start by taking the lift $\bar s$ of $s$ starting at $\bar r(1)$. Since $rs = c_1(\alpha,h)$ we have that $\bar r\bar s = \bar c_1(\alpha,h)$, and thus $\bar s\subset \Lambda_{\alpha}/G_h$ and $\bar s(1) = \bar\gamma_h(1)$. We continue lifting $c_1(\alpha,g)c_2(\alpha,g) = stur$ by taking the lift of $tu$ beginning at $\bar s(1) =\bar\gamma_h(1)$, and we notice that this lift agrees with $\bar c_2(\alpha,h)$ since $tu=c_2(\alpha,h)$. In particular it ends at $\bar\gamma_h(1)$, so $\bar r$ is the lift of $r$ that we need to complete the lifting of $c_1(\alpha,g)c_2(\alpha,g) = stur$. Thus we get that $\bar c_1(\alpha,g)\bar c_2(\alpha,g)=\bar s \bar c_2(\alpha,h) \bar r $, which is closed at $\bar r(1)$ and clearly contained in $\Lambda_{\alpha}/G_h$.

{\sl Claim 2:} For $i=1,2$ we have
\begin{equation}\label{firsteq}
\bar c_i(\alpha,g)=\bar c_i(\alpha,h).
\end{equation}

We can consider $\mathcal L_{\bar\alpha(0)}(\Lambda_{\alpha}/G_h)$ as a subgroup of $\mathcal L_x(M)$, since projection induces an injective homomorphism, and then Lemma \ref{quotient} implies that $T_1(\bar\alpha)=\{h,h^{-1}\}$. By construction we have $$ \bar c_i(\alpha,h)= c_i(\bar\alpha,h) \mbox{ for } i=1,2.$$

On the other hand, by Claim 1 we have that $\bar c_1(\alpha,g)$ and $\bar c_2(\alpha,g)$ lie inside $\Lambda_{\alpha}/G_h$, and we recall that they are lifts of $\alpha_g$ and $\beta_g$ respectively, which are permutations of $\alpha$ and $\alpha^{-1}$. Therefore  $\bar c_1(\alpha,g)$ and $\bar c_2(\alpha,g)$ are permutations of $\bar\alpha$ and $\bar\alpha^{-1}$ respectively (in particular they are closed). 
Then we get that $\bar c_1(\alpha,g)\bar c_2(\alpha,g)$ represents the dot product for an intersection pair of $\bar\alpha$, and since it is cyclically reduced we must have  $$\bar c_i(\alpha,g)= c_i(\bar\alpha,h^{\epsilon}) \mbox{ for } i=1,2 \mbox{ and for some } \epsilon=\pm 1 $$

Since we have $\bar c_1(\alpha,h)\bar c_2(\alpha,h)=\bar r \bar c_1(\alpha,g)\bar c_2(\alpha,g) \bar r^{-1}$, we can apply Remark \ref{dot product of inverses} to get that $\epsilon=1$, proving this claim.

To finish the proof of the lemma recall that, since we are in the orientation reversing case, we have  $c_1(\alpha,g)=\alpha_g$ and $c_1(\alpha,h)=\alpha_h$, which are permutations of $\alpha$. Equation \eqref{firsteq} implies, by projecting, that  $\alpha_g=\alpha_h$, and then Lemma \ref{primitive} gives $g=h$, since $\alpha$ is primitive.

\end{proof}

Putting Lemma \ref{reversing2} together with Remark \ref{r:int length}, we see that a dot product of the form $(X\cdot_{(P,Q)}X^{-1})$ with length $2l(\alpha)$ comes from a unique intersection pair $(P,Q)$, which is of type either (1) or (3). 

One would like to remove the orientation reversal condition from the hypothesis of Lemma \ref{reversing2}, for that would give a stronger result than Theorem \ref{simple}, namely that the strings in the terms of the formula for $[X,X^{-1}]$ cannot repeat. Unfortunately this remains open. Next is the result we can get when dropping said orientation condition, which will suffice for our purpose.

\begin{lemma}\label{general1} Assume that $\alpha$ is primitive, and that $g, h\in T_1(\alpha)$ are such that $[\alpha]h[\alpha]^{-1}h^{-1}$ and $[\alpha]g[\alpha]^{-1}g^{-1}$
are conjugate. Then 
\[ c_i(\alpha,g)= c_i(\alpha,h) \qquad \mbox{ for }i=1,2.\]

%\[ c_1(\alpha,g)= c_1(\alpha,h) \]
\end{lemma}

\begin{proof}

We shall assume $g$ and $h$ preserve orientation strictly, since Remark \ref{r:int length} implies that the only other possible case is the one covered by Lemma \ref{reversing2}. 

Recall the notation from Equations \eqref{c1h} and \eqref{c1g}, i.e. the curves $r,s,t, u$ and their properties. Notice that we can prove the lemma by showing that $t$, and hence $r$, are constant. 

Firstly we shall define some horizontal lifts in $\mathcal E_x(M)$ that will be useful through the proof: By Equation \eqref{c1g} we can take $\bar s$ the lift of $s$ starting at $\tilde \gamma_g(1)=b_g(1)$, and $\bar t$ the lift of $t$ starting at $\bar s(1)$. Note that $\bar s\bar t$ is the lift of $c_1(\alpha,g)$ that is an initial segment of $\tilde{\alpha}_g$, thus it is contained in $\Lambda_{\alpha}$ and ends at $\bar t(1)=[\alpha]b_g(0)$, recalling Definition \ref{d:c1c2} in the orientation preserving case. We also see that $\bar s$ and $\bar t$ are positively oriented in $\Lambda_{\alpha}$. 

We also define $\tilde r$ as the lift of $r$ starting at $\tilde\gamma_h(1)=b_h(1)$, and $\tilde s$ as the lift of $s$ starting at $\tilde r(1)$, which are well defined by Equation \eqref{c1h}. Again we see that $\tilde r\tilde s$ is the lift of $c_1(\alpha,g)$ which is an initial segment of $\tilde \alpha_h$, and so it is contained in $\Lambda_{\alpha}$ and ends at $\tilde s(1)=[\alpha]b_h(0)$. Also, $\tilde r$ and $\tilde s$ are positively oriented in $\Lambda_{\alpha}$. 

Finally, let $\tilde t$ be the lift of $t$ starting at $\tilde s(1)$, which is well defined since $t(0)=s(1)$ from Equation \eqref{c1g}. We see from Equation \eqref{c1h} that $\tilde t$ is an initial segment of a lift of $c_2(\alpha,h)$, namely the one contained in $[\alpha]\tilde\beta_h$, since it starts at $[\alpha]b_h(0)$. Thus $\tilde t$ is contained in $[\alpha]h\Lambda_{\alpha}$, meeting $\Lambda_{\alpha}$ only at $\tilde t(0)$, and it goes in the negative direction with respect to the orientation of $[\alpha]h\Lambda_{\alpha}$ induced by  $[\alpha]h$. 

Next we recall Equation \eqref{8}, defining $\phi=[\gamma_h r \gamma_g^{-1}]$. Notice that, by the above definitions, we have $\phi \tilde \gamma_g(1) = \tilde r(1)$, in particular $\phi\in T(\Lambda_{\alpha})$. Since the action of $\mathcal L_x(M)$ preserves horizontal lifting, we also get that $\phi \bar s = \tilde s$ and $\phi \bar t= \tilde t$.

{\sl Claim 1}: If $\phi\in T^+(\Lambda_{\alpha})$ then $t$ is constant.

Let $G$ be the subgroup generated by $g,h,\phi$ and $[\alpha]$. Then $G$ is $\alpha$-oriented by Lemma \ref{oriented}, so there is a $G$-invariant orientation on $G\Lambda_{\alpha}$. On the other hand, $\phi \bar t=\tilde t$ where $\bar t$ has positive orientation in $\Lambda_{\alpha}$ but $\tilde t$ has negative orientation in $[\alpha]h\Lambda_{\alpha}$. This is a contradiction unless $t$ is constant.

{\sl Claim 2}: If $s$ is non-constant, then $\phi\in T^+(\Lambda_{\alpha})$.

We have $\phi \bar s=\tilde s$ where both $\bar s$ and $\tilde s$ are contained in $\Lambda_{\alpha}$, thus $\bar s\subset b_{\phi}(I)$ and $\tilde s\subset a_{\phi}(I)$. Recall that both $\bar s$ and $\tilde s$ are positively oriented in $\Lambda_{\alpha}$, thus showing that $\phi$ preserves orientation, if $s$ is non-constant.

Recall that we prove the lemma by showing that $t$ is constant. Thus, in light of these claims, the only case that remains to be considered is when $s$ is constant and $\phi$ reverses orientation strictly. We shall show that this case is void, which makes sense as $s$ and $t$ cannot be both constant in Equation \eqref{c1g}. Thus we assume that $s$ is constant and $\phi$ strictly reverses orientation, aiming to reach a contradiction.

This case is the most complex part of this proof, the key will be to consider the dot product defined by $\phi$. Note that this can be defined even if $\phi$ is not in $T_1(\alpha)$, through Lemma \ref{mn}, and the same is true for the curves $c_1(\alpha,\phi)$ and $c_2(\alpha,\phi)$ from Definition \ref{d:c1c2}. In order to simplify notation we put $c_i(\phi)=c_i(\alpha,\phi)$ for $i=1,2$, and define  $$ \omega = [\alpha] \phi [\alpha]^{-1}\phi^{-1} = [\gamma_{\phi}c_1(\phi)c_2(\phi)\gamma_{\phi}^{-1} ] $$

{\sl Claim 3:} $\omega$ is primitive.

By change of basepoint, i.e. Lemma \ref{bpoint}, this is equivalent to show that $c_1(\phi)c_2(\phi)$ is primitive. Assume the contrary, i.e. that there is a closed curve $\tau$ and $k>1$ so that $\tau^k=c_1(\phi)c_2(\phi)$. Note that $\tau$ must be cyclically reduced, since $c_1(\phi)c_2(\phi)$ is, and that $l(c_1(\phi))=l(c_2(\phi))$ by Remark \ref{pair2}.

 If $k$ is even, we write $k=2j$ and get that $c_1(\phi)=\tau^j=c_2(\phi)$, which is absurd by Remark \ref{pair2} and the assumption that $\phi$ reverses orientation, since a permutation of $\alpha$ cannot agree with a permutation of $\alpha^{-1}$. If $k$ is odd, we write $k=2j+1$ and we get that $$\tau=vw  \,\,\mbox{ with }\,\, c_1(\phi)=(vw)^jv \,\,\mbox{ and }\,\, c_2(\phi)=w(vw)^j $$ where we also have $l(v)=l(w)$. Again by Remark \ref{pair2}, we get that $(w^{-1}v^{-1})^jw^{-1}$ is a permutation of $(vw)^jv$. Since $l(v)=l(w)$ and $v$ cannot overlap $v^{-1}$ (nor $w$ overlap $w^{-1}$), we must have $v=w^{-1}$, which is absurd since $\tau$ cannot be trivial. Thus we have shown Claim 3.

{\sl Claim 4:} There is a non-trivial $\theta\in T(\Lambda_{\alpha})$ such that $\theta\omega\theta^{-1}=\omega$.

To show this claim it will be useful to write Equations \eqref{c1h} and \eqref{c1g} for the case when $s$ is constant:
\begin{equation}\label{c1c2 s cte}
c_1(\alpha,h)=c_2(\alpha,g)=r \,\, \mbox{ and }\,\, c_1(\alpha,g)=c_2(\alpha,h)=t,
\end{equation}  
and we also get that $\bar t(0) = \tilde \gamma_g(1)=b_g(1)$ and $\tilde r(1)=\tilde t(0)=[\alpha]b_h(0)$. By Remark \ref{rm:inv c1c2} we deduce that 
\begin{equation}\label{c1c2 -1 s cte}
c_1(\alpha,h^{-1})=c_2(\alpha,g^{-1})=t^{-1} \,\, \mbox{ and }\,\, c_1(\alpha,g^{-1})=c_2(\alpha,h^{-1})=r^{-1}
\end{equation}  

Define $$\psi = h^{-1}\phi g. $$
Let us check that $\psi\in T^-(\Lambda_{\alpha})$.  
Note first that, since $g$ preserves orientation, we have $\tilde\gamma_{g^{-1}}(1)=b_{g^{-1}}(1)=a_g(1)$ by Remark \ref{inverse}. By Equation \eqref{c1c2 -1 s cte} we may take $\hat r$ the lift  of $r$ that ends at $\hat r(1)=a_g(1)$, i.e. so that $\hat r^{-1}$ lifts $c_1(\alpha,g^{-1})$ starting at $a_g(1)=\tilde\gamma_{g^{-1}}(1)$. Thus we see that $\hat r^{-1}$ is contained in $\Lambda_{\alpha}$ and is positively oriented, since it is a segment of $\tilde \alpha_{g^{-1}}$. We also see that $\hat r(0)=\hat r^{-1}(1) = [\alpha]a_g(0)$. 

We shall describe the action of $\psi = h^{-1}\phi g$ on $\hat r$. First we notice that $g\hat r$ ends at $ga_g(1)=b_g(1)=\tilde\gamma_g(1)$. Next we recall that $\phi \tilde\gamma_g(1)=\tilde r(1)$, so by equivariance of the horizontal lifting we get that $\phi g\hat r=\tilde r$. Finally we get that $\psi\hat r =h^{-1}\tilde r$. Recalling Equation \eqref{c1c2 -1 s cte} we see that this curve lifts $c_2(\alpha,h^{-1})^{-1}$ starting at $h^{-1}\tilde r(0)=h^{-1}b_h(1)=a_h(1)$. Thus $\psi\hat r$ is contained in $h^{-1}\Lambda_{\alpha}$ and meets $\Lambda_{\alpha}$ exactly at $\psi\hat r(0)=a_h(1)$. This shows that $\psi\in T(\Lambda_{\alpha})$, and it must reverse orientation strictly by Lemma \ref{oriented} since $\phi=h\psi g^{-1}$.

Now we set $h^{\ast}=[\alpha]h[\alpha]^{-1}$, and a simple computation from Equation \ref{phi conj} and the definition of $\psi=h^{-1}\phi g$ gives us 
\begin{equation} \label{dot pr psi} h^{\ast} [\alpha] \psi [\alpha]^{-1}\psi^{-1}{h^{\ast}}^{-1}=[\alpha] \phi [\alpha]^{-1}\phi^{-1} 
\end{equation}
By  Remark \ref{mn} there are  integers $i, j, k, l$ 
such that   $[\alpha]^i\phi[\alpha]^j$ and  $[\alpha]^k\psi[\alpha]^l$ belong to $T_1(\alpha)$. Now we apply Lemma \ref{reversing2} to these elements and Equation \eqref{dot pr psi}, hence for $n=i-k$, $m=k-l$ we obtain
$$
\psi=[\alpha]^n\phi[\alpha]^m 
$$
and Equation \eqref{dot pr psi} becomes 
$$
 (h^{\ast}  [\alpha]^n)[\alpha]\phi [\alpha]^{-1}\phi^{-1}({h^{\ast}}[\alpha]^n)^{-1}=[\alpha] \phi [\alpha]^{-1}\phi^{-1} 
$$
So we set $\theta=h^{\ast}[\alpha]^n$ and get that $\theta\omega\theta^{-1}=\omega$. Note that $\theta=[\alpha]h[\alpha]^{n-1}$ is not trivial, since $h$ is not a power of $[\alpha]$, and belongs to $T(\Lambda_{\alpha})$, since $\theta\Lambda_{\alpha}=[\alpha]h\Lambda_{\alpha}$ intersects $\Lambda_{\alpha}$ in the segment $[\alpha]b_h(I)$. Thus we have shown Claim 4.

Putting $C=\alpha(I)$ and recalling Lemma \ref{lem:ident}, we see that $\omega,\theta\in \mathcal L_x(C)$ which is a free group, so Claims 3 and 4 imply that 
\begin{equation} \label{c-phi}
\theta=\omega^i \,\,\mbox{ for some } i\in\Z, i\neq 0
\end{equation} We shall reach a contradiction by showing that $\theta\Lambda_\alpha$ is disjoint from $\Lambda_{\alpha}$, i.e.that $\theta\notin T(\Lambda_{\alpha})$, against Claim 4. Since $T(\Lambda_{\alpha})$ is closed under taking inverses (Remark \ref{inv T}), we may assume that $i>0$.

Changing the basepoint if necessary, we may assume that $b_{\phi}(1)=id_x$. Note that such a change of basepoint ammounts to repalce $\alpha$ by a permutation of it, and does not change the curves $c_1(\phi)$ and $c_2(\phi)$. By Lemma \ref{bpoint} this change of basepoint also preserves Claims 3 and 4, and Equation \eqref{c-phi}. From $b_{\phi}(1)=id_x$ we get that $\gamma_{\phi}$ is constant, $c_1(\phi)=\alpha$ and $c_2(\phi)$ is a non-trivial permutation of $\alpha^{-1}$. 

Let $\tilde c_2(\phi)$ be the horizontal lift of $c_2(\phi)$ that ends at $b_{\phi}(1)=id_x$. Since $\phi$ reverses orientation, the discussion after Definition \ref{d:c1c2} implies that $\tilde c_2(\phi)$ intersects $\Lambda_{\alpha}$ exactly in the segment $b_{\phi}(I)$, and in particular $\tilde c_2(\phi)(0)\notin \Lambda_{\alpha}$. On the other hand we take $\eta$ the horizontal lift of $(c_1(\phi)c_2(\phi))^i$ that begins at $id_x$, so we have $\eta(1)=\omega^i=\theta$. Note that since $c_1(\phi)=\alpha$, we have a reduced factorization $$\eta=\tilde\alpha\nu$$ where $\nu$ meets $\Lambda_{\alpha}$ only at $\nu(0)$, since the lift of $c_2(\phi)$ starting at $[\alpha]=[\alpha]b_{\phi}(1)$ is an initial segment of $\nu$. Observe also that $\theta\tilde c_2(\phi)$, being the lift of $c_2(\phi)$ that ends at $\theta=\nu(1)$, is a final segment of $\nu$ (if $i=1$ it agrees with the initial segment just discussed). Since $\mathcal E_x(C)$ is a tree, the line $\theta\Lambda_{\alpha}$ cannot intersect $\Lambda_{\alpha}$ without containing $\nu(I)$, but acting by $\theta$ we see that $\theta\tilde c_2(\phi)$ intersects $\theta\Lambda_{\alpha}$ only at $\theta b_h(I)$, which does not contain $\theta \tilde c_2(\phi)(0)\in\nu(I)$. Therefore $\theta\notin T(\Lambda_{\alpha})$ and we have a contradiction, concluding the proof.

\end{proof}

\section{Signs of the terms of the bracket}

Here we shall study the signs of the linked pairs in $LP_2(X,X^{-1})$ for $X\in S(M)$, showing that linked pairs yielding the same dot product have also the same sign, and finally arriving at the proof of Theorem \ref{simple}. In order to do so, we may need to consider a small deformation of a curve $\alpha$ representing $X$.

\subsection{Deformations of 1-complexes}

Let $C$ be a one dimensional complex, and consider $p\in C$. Then a small enough neighborhood $B$ of $p$ in $C$ is homeomorphic to a wedge of intervals, each one with an endpoint at $p$ and the other in $\partial B$. We call the {\em valence} of $p$ in $C$ to the number of such segments. A point of valence $2$ will be called a {\em regular point} of $C$. Observe that the set of non-regular points is discrete, while the components of the set of regular points are arcs.

We will be interested in complexes of the form $C=\alpha(I)$ where $\alpha\in \Omega$ is a cyclically reduced, non-constant closed curve. Then $C$ has no points of valence $1$, and by compactness, the set of non-regular points (i.e. {\em branching points}) is finite. The closure of a component of the set of regular points is a segment with endpoints at non-regular points. Note that, replacing $\alpha$ by a permutation if necessary, we can assume that $x=\alpha(0)$ is a regular point.

Before introducing the perturbation of $\alpha$ that we need, it is worth recalling that $C_{\varepsilon}$ is the $\varepsilon$-neighborhood of $C$ in $M$, and for $\varepsilon$ small enough, Lemma \ref{tubular} states that $C_{\varepsilon}$ retracts by deformation onto $C$. Then the map induced by the inclusion $i_\ast: \pi_1(C, x)\rightarrow \pi_1( C_\varepsilon, x)$ is an isomorphism, and  its inverse is $\chi_\ast: \pi_1(C_\varepsilon, x)\rightarrow \pi_1( C, x) $, which is induced by a retraction $\chi:C_{\varepsilon}\to C$.

\begin{lemma}\label{approximation}
Let $\alpha$ be a cyclically reduced closed curve and $C=\alpha(I)$. Assume that $x=\alpha(0)$ is a regular point of $C$. %Let  $C_{\varepsilon}$ be the  $\varepsilon$ neighborhood  of $C=\alpha(I)$
Then there are $\varepsilon>0$ and a closed, cyclically reduced curve $\gamma\subset C_{\varepsilon}$ such that 
\begin{enumerate}
\item $\gamma$ is an $\varepsilon$-perturbation of $\alpha$, with $\gamma(0)=x$,
\item $\Gamma=\gamma(I)$ has no points of valence greater than $3$, and
\item The map $i_\ast: \pi_1(\Gamma, x)\rightarrow \pi_1( C_\varepsilon, x)$ induced by the inclusion is an isomorphism.
\end{enumerate}

Moreover, $\gamma$ only differs from $\alpha$ in an $\varepsilon$-neighborhood of the non-regular points of $C$.

%$\gamma\subset C_\varepsilon$ based at $x$, such that every point in  $\gamma(I)$ has multiplicity not greater than $3$, and the map induced by the inclusion  $i_\ast: \pi_1(\gamma(I), x)\rightarrow \pi_1( C_\varepsilon, x)$ is an isomorphism.  
\end{lemma}

\begin{proof}

Set $\varepsilon$ as in Lemma \ref{tubular}, though we may need to reduce it further. For each non-regular point $p$ of $C$ we consider $B_p$ a closed ball in $M$ centered at $p$ with radius $\varepsilon/2$. Note that if $\varepsilon$ is small enough, then $B_p$ is a normal ball and $B_p\cap C$ is a union of geodesic segments joining $p$ to $\partial B_p$. Reducing $\varepsilon$ if necessary, we can also make the sets $B_p$ pairwise disjoint, and disjoint from $x$ by our assumption.

First we will construct the complex $\Gamma$. For each non-regular point $p$ of $C$, we take a segment $\eta\subset\alpha$ with endpoints in $\partial B_p$ and so that $\eta(I)\subset B_p$, i.e. $\eta$ is a segment of $\alpha$ that traverses $B_p$. Then we consider a piecewise geodesic complex $Y_p$ contained in $B_p$, with $Y_p\cap\partial B_p = C\cap\partial B_p$, and so that $Y_p$ is a finite tree containing $\eta(I)$, whose branching points lie in $\eta(I)$ and have valence $3$. Figure \ref{fig:6} shows an example of this construction. It can be interpreted as a deformation of $C\cap B_p$ that spreads out the segments that are not in $\eta$, so that their endpoints are spread along $\eta$ instead of converging at $p$. Then we define $\Gamma$ so that it agrees with $C$ in the complement of $B=\bigcup_p B_p$, and that $\Gamma\cap B_p=Y_p$ for every non-regular point $p$ in $C$.

\begin{figure}[htbp]
\input{fig9.pic}
\caption{proof of lemma \ref{approximation}}\label{fig:6}
\end{figure}

By construction, $\Gamma$ is a connected piecewise geodesic complex and all its non-regular points are of valence $3$. Note also that $B_p$ retracts by deformation to $Y_p$ for each $p$. These maps can be extended to a retraction by deformation $\chi_1:C_{\varepsilon}\to\Gamma$, thus obtaining point (3) in the statement.

Now we describe the curve $\gamma$. Write 
\begin{equation}\label{alpha}
\alpha=\alpha_0\beta_1\alpha_1\cdots \beta_{n}\alpha_n
\end{equation}
where $\alpha_i$ is contained in the closure of the complement of $B$ for all $i=0,\ldots,n$, and $\beta_j$ is contained in $B$, for $j=1,\ldots,n$. None of these curves is constant since $x$ is outside $B$.
Then for each $j$ there is some $p$ with $\beta_j\subset B_p$, and we have $\beta_j(0),\beta_j(1)\in\partial B_p$. Note that, by construction, there is a unique reduced curve $\bar \beta_j$ joining $\beta_j(0)$ to $\beta_j(1)$ in $Y_p$. Thus $\bar\beta_j\subset \Gamma\cap B_p$ with  $\bar\beta_j(0)=\beta_j(0)$ and  $\bar\beta_j(1)=\beta_j(1)$. We define 
\begin{equation}\label{gamma}
\gamma=\alpha_0\bar\beta_1\alpha_1\cdots \bar\beta_{n}\alpha_n
\end{equation}
i.e. we replace each $\beta_j$ with $\bar\beta_j$. This is an $\varepsilon$-perturbation of $\alpha$, since for each $j$, $\beta_j$ and $\bar\beta_j$ are in the same ball of radius $\varepsilon/2$. It is also clear that $\gamma$ admits no reductions, and that $\gamma(0)=\gamma(1)=x$. It only remains to show that $\gamma(I)=\Gamma$. By construction we have $\gamma(I)\subseteq\Gamma$, and it is also clear that the closure of $\Gamma-B$ is contained in $\gamma(I)$, writing this set as $\bigcup_i\alpha_i(I)$. So we must show that $Y_p\subset\gamma(I)$ for each non-regular point $p$ of $C$. For one such $p$ we note that the curve $\eta$ used in the construction of $Y_p$ is in $\gamma(I)$: by its definition $\eta=\beta_j$ for some $j$ (maybe more than one), and in that case we also have $\bar\beta_j=\eta$. On the other hand, since we have $\partial B_p\cap\Gamma\subset\gamma(I)$, the rest of the segments making up $Y_p$ must also be contained in $\gamma(I)$.

\end{proof}

\subsection{Signs and lifts}\label{liftsign}

Now we shall interpret the signs of linked pairs in terms of horizontal lifting, in the same fashion as we did for intersection pairs and dot products in the previous sections. Let $\alpha\in\Omega$ be a non-trivial cyclically reduced closed curve and $x=\alpha(0)$.

\begin{definition} \label{def:liftsign} For $g\in T_0(\alpha)$ we write $$\epsilon_g(\alpha)=sign(P_g,Q_g), $$ i.e. the sign of the linked pair of $\alpha$ corresponding to $g$. When the curve $\alpha$ is clear from the context, we just write $\epsilon_g$.
\end{definition}

Let $C=\alpha(I)$ and recall from Lemma \ref{lem:ident} that the universal cover $\tilde C$ is isomorphic, as a principal fiber bundle, to $\mathcal E_x(C)$. We choose such an isomorphism by picking  $\tilde x\in \tilde C$ a lift of $x$, and setting that $id_x$ corresponds to $\tilde x$. Then we can identify $\Lambda_{\alpha}$ with the infinite lift of $\alpha$ to $\tilde C$ that starts at $\tilde x$. 

Let $C_{\varepsilon}$ be a neighborhood of $C$ satisfying Lemma \ref{tubular}. Then its universal cover $\tilde C_{\varepsilon}$ contains $\tilde C$, and retracts onto it by lifting the retraction $\chi:C_{\varepsilon}\to C$. Note that the complement of $\Lambda_{\alpha}$ in $\tilde C_{\varepsilon}$ has two connected components. We give $\tilde C_{\varepsilon}$ the orientation lifted from that of $C_{\varepsilon}\subset M$, and let $C^+(\alpha)$ be the {\em left side} of $\Lambda_{\alpha}$. More precisely, $C^+(\alpha)$ is the component of $\tilde C_{\varepsilon}-\Lambda_{\alpha}$ so that the standard orientation of $\Lambda_{\alpha}$, given by $\tilde\alpha$, agrees with the orientation induced on $\Lambda_{\alpha}$ as part of the boundary  $\partial C^+(\alpha)$. The other component, i.e. the {\em right side} of $\Lambda_{\alpha}$, will be denoted by $C^-(\alpha)$.

Let $a_1$, $a_2$, $\nu$ be small geodesic segments in $\tilde C_{\varepsilon}$, so that $a_1(1)=a_2(0)$ and $\nu$ meets $a_1a_2$ only at this point, which is an endpoint of $\nu$. Note that we have $sign(a_1a_2,\nu)=sign(\pi(a_1a_2),\pi(\nu))$, by definition of the orientation of $\tilde C_{\varepsilon}$.
In case that $a_1a_2$ is contained in $\Lambda_{\alpha}$ with positive orientation, then we have $sign(a_1a_2,\nu)=1$ if either $\nu(0)=a_1(1)$ and $\nu(1)\in C^+(\alpha)$, or $\nu(1)=a_1(1)$ and $\nu(0)\in C^-(\alpha)$. In the reverse cases we have $sign(a_1a_2,\nu)=-1$.

Now consider $g\in T_1(\alpha)$. We identify $\mathcal L_x(C)$ with $\pi_1(C_{\varepsilon},x)$ acting on $\tilde C_{\varepsilon}$ by deck transformations, using Lemmas \ref{lem:ident} and \ref{tubular}. Then, by definition of $T_1(\alpha)$, we have that $g[\alpha]\tilde x$ and $g[\alpha^{-1}]\tilde x$ are not contained in $\Lambda_{\alpha}$. Due to the previous observations, we see that $g$ corresponds to a linked pair, i.e. $g\in T_0(\alpha)$, iff $g[\alpha]\tilde x$ and $g[\alpha^{-1}]\tilde x$ are in different components of $\tilde C_{\varepsilon}-\Lambda_{\alpha}$. In that case, we have $\epsilon_g=1$ if $g[\alpha^{-1}]\tilde x\in C^-(\alpha)$, and thus $g[\alpha]\tilde x\in C^+(\alpha)$, which is to say that $g\Lambda_{\alpha}$ crosses $\Lambda_{\alpha}$ going from its right side and towards its left side. We have $\epsilon_g=-1$ if the reverse holds. We compile these results for future reference 

\begin{remark} \label{rm:signs} Let $g\in T_1(\alpha)$, then
\begin{itemize}
\item $g\in T_0(\alpha)$ iff $g[\alpha]\tilde x$ and $g[\alpha^{-1}]\tilde x$ are on different sides of $\Lambda_{\alpha}$.
\item In that case, $\epsilon_g=1$ iff $g[\alpha^{-1}]\tilde x\in C^-(\alpha)$, i.e. is at the right side of $\Lambda_{\alpha}$.
\end{itemize}
\end{remark}

In a similar manner, we see that for $g\in T_0(\alpha)$ we have $\epsilon_g=1$ iff $\tilde\beta_g(1)\in C^-(\alpha)$. Equivalently, iff $g[\alpha]g^{-1}\tilde\beta_g(0)\in C^+(\alpha)$. Noting that $c_2(\alpha,g)$ ends at $t_g(1)$ by Definition \ref{d:c1c2}, we have the following.

\begin{remark} \label{rm:signs2} For $g\in T_0(\alpha)$ let  $\tilde c_2(\alpha,g)$ be the lift of $c_2(\alpha,g)$ that ends at $b_g(1)$ (namely, the one contained in $g[\alpha]g^{-1}\tilde\beta_g$). Then $\epsilon_g=1$ iff $\tilde c_2(\alpha,g)$ is at the left of $\Lambda_{\alpha}$ (more precisely, is contained in the closure of $C^+(\alpha)$).
\end{remark} 

By the comments after Definition \ref{d:c1c2} we see that the intersection of $\tilde c_2(\alpha,g)$ with $\Lambda_{\alpha}$ is either the endpoint $b_g(1)$, in case $g$ preserves orientation, or the segment $b_g(I)$, if $g$ reverses orientation. We also point out that the terminology of orientation preserving or reversing elements of Definition \ref{def:orient}  does not relate to the orientation of $\tilde C_{\epsilon}$, which is preserved by every deck transformation. 

%Let $\tilde c_2(\alpha,g)$ be the lift of $c_2(\alpha,g)$ that is contained in $\tilde\beta_g$, i.e. the one beginning at $b_g(0)$ if $g$ preserves orientation, or at $b_g(1)$ if $g$ reverses orientation. Then we get that $\epsilon_g=1$ iff $\tilde c_2(\alpha,g)$ lies at the right of $\Lambda_{\alpha}$ (i.e. is contained in $C^-(\alpha)$ aside from the endpoint in $\Lambda_{\alpha}$).

%boundary orientation of $\partial C^+(\alpha)$ gives $\Lambda_{\alpha}\subset\partial C^+(\alpha)$ the same orientation as the one given by $\tilde\alpha$.  

\subsection{Signs and deformations}

Let $\alpha\in\Omega$ be a cyclically reduced closed curve, and assume that $x=\alpha(0)$ is a regular point of $C=\alpha(I)$. Take $\varepsilon$ and $\gamma$ as given by Lemma \ref{approximation}, and let $\Gamma=\gamma(I)$. Then we can identify $\mathcal E_x(\Gamma)$ with $\tilde \Gamma$, the universal cover of $\Gamma$, which is embedded in $\tilde C_{\varepsilon}$. This identifies $\mathcal L_x(\Gamma)$ with $\pi_1(C_{\varepsilon},x)$, and thus with  $\mathcal L_x(C)$. We will be assuming these identifications in the sequel, though we should make clear  that $\mathcal L_x(C)$ and $\mathcal L_x(\Gamma)$ are different as subgroups of $\mathcal L_x(M)$, and that $\mathcal E_x(C)\cup \mathcal E_x(\Gamma)$, as a subspace of $\mathcal E_x(M)$, is not homeomorphic to $\tilde C\cup\tilde\Gamma$. In fact, this correspondence identifies $[\gamma]\in \mathcal L_x(\Gamma)$ with $[\alpha]\in\mathcal L_x(C)$. 

We also identify $\Lambda_{\gamma}$ with the infinite lift of $\gamma$ at $\tilde x$, the lift of $x$ used to define both correspondences $\mathcal E_x(C)\cong \tilde C$ and $\mathcal E_x(\Gamma)\cong\tilde\Gamma$. Thus the sets $T(\Lambda_{\alpha})$ and $T(\Lambda_{\gamma})$ can be considered as subsets of the same group $\pi_1(\tilde C_{\varepsilon},x)$, which we shall see as the group of deck transformations of $\tilde C_{\varepsilon}$.

%{\bf Comentar:} %Lx(C) y Lx(Gamma) no iguales, pero se identifican por isomorfismo. \tilde C y \tilde Gamma homeo a Ex(C) y Ex(Gamma), pero las uniones no son homeo. Ver Lambda_a Lambda_g en \tilde C_e y T_1(a) T_1(g) como subconjuntos de pi_1(C_e) y/o deck transformations. 

%For the sequel we identify   $\pi_1(\alpha(I), x)$    and   $\pi_1(\gamma(I), x)$ with their image in  $\pi_1( C_\varepsilon, x)$. 

%Let $C_{\varepsilon}$ as in subsection \ref{dif}. Let $\tilde{C_{\varepsilon}}$ be the universal cover of $C_{\varepsilon}$. By Lemma \ref{lem:ident}  $\mathcal E_x(C)=\tilde C \hookrightarrow \tilde{C_{\varepsilon}}$. Therefore $T(\Lambda_{\alpha})\Lambda_{\alpha}$ is injectively included  in  $\tilde{C_{\varepsilon}}$. 

%{\bf ---------usar la def $sign(ab,c)$ de subs \ref{LPs} ----------}

%Let $(P, Q)$ be a linked pair corresponding to $g\in T_0(\alpha)$. Let $\tilde c_2(\alpha, g)$ be the lift starting at $b_g(1)$. Then $sign(P, Q)=1$ if and only if $\tilde c_2(\alpha, g)$ lies at the right of $\Lambda_{\alpha}$. 

%{\bf ------ Mostrar $\mathcal L_x$ preserva orientaci\'on 2D (en cada v\'ertice): se extienden a deck transformations--------------}

%Recall the definition of multiplicity in \ref{multiplicity}. 

\begin{lemma}\label{sign}
Let $\alpha$ and $\gamma$ be as in Lemma \ref{approximation}, and take $g\in \pi_1(\tilde C_{\varepsilon},x)$. Then $g\in T_0(\alpha)$ iff $g\in T_0(\gamma)$, in which case $\epsilon_g(\alpha)=\epsilon_g(\gamma)$. 
\end{lemma}

\begin{proof}

Let $B$ be the $\varepsilon$-neighborhood of $\Lambda_{\alpha}$ in $\tilde C_{\varepsilon}$, which also contains $\Lambda_{\gamma}$. Moreover, we have $C^{\pm}(\alpha)\cap(\tilde C_{\varepsilon}-B)=C^{\pm}(\gamma)\cap(\tilde C_{\varepsilon}-B)$, i.e. points outside $B$ are in the same side with respect to both $\Lambda_{\alpha}$ and $\Lambda_{\gamma}$. Note that the construction of Lemma \ref{approximation} allows for reducing $\varepsilon$ as necessary, so we can assume that the translates of $\tilde x$ that do not belong to $\Lambda_{\alpha}$ are outside $B$. 

If $g\in T_0(\alpha)$ we observed in Remark \ref{rm:signs} that $g[\alpha^{-1}]\tilde x$ and $g[\alpha]\tilde x$ are on different sides of $\Lambda_{\alpha}$, and since they are not in $B$, they are also on different sides of $\Lambda_{\gamma}$. Since $[\alpha]$ and $[\gamma]$ agree when seen as elements of $\pi_1(\tilde C_{\varepsilon},x)$, we see that $g\Lambda_{\gamma}$ contains both $g[\alpha^{-1}]\tilde x$ and $g[\alpha]\tilde x$, implying that $g\Lambda_{\gamma}$ and $\Lambda_{\gamma}$ intersect. This intersection corresponds to a linked pair by Remark \ref{rm:signs}. We get that $g\in T_1(\gamma)$ by recalling the definition of this set, together with the fact that $\alpha$ and $\gamma$ agree on a neighborhood of their basepoint, so every lift of $\gamma$ is an $\varepsilon$-perturbation of the corresponding lift of $\alpha$ that agrees with it in a neighborhood of its endpoints. With this we conclude that $g\in T_0(\gamma)$.

%the identification between $\mathcal L_x(C)$ and $\mathcal L_x(\Gamma)$.    

The reciprocal argument is analogous. We have that $\epsilon_g(\alpha)=\epsilon_g(\gamma)$ by Remark \ref{rm:signs}, together with the facts that $C^{\pm}(\alpha)\cap(\tilde C_{\varepsilon}-B)=C^{\pm}(\gamma)\cap(\tilde C_{\varepsilon}-B)$ and that $g[\alpha^{-1}]\tilde x$ and $g[\alpha]\tilde x$ lie outside $B$.

\end{proof}

Let $X$ and $Y$ be the strings represented by $\alpha$ and $\gamma$ as in Lemma \ref{approximation}. Then Lemma \ref{sign} gives a bijection between $LP_2(X,X^{-1})$ and $LP_2(Y,Y^{-1})$ that preserves the sign. It also preserves the dot product, by the isomorphism between $\mathcal L_x(C)$ and $\mathcal L_x(\Gamma)$ and Lemma \ref{formula2}. Thus proving Theorem \ref{simple} for $Y$ also implies it for $X$, i.e. we may replace $\alpha$ with $\gamma$ whenever necessary in our proof. We will only be doing this replacement at the points of the argument that require it, namely in the next Lemma \ref{general2}.

The proof of Lemma \ref{sign} clearly does not generalize for intersection pairs that are not linked. It is also possible to show Lemma \ref{sign} by following what happens to an intersection pair during the construction of $\gamma$ in Lemma \ref{approximation}, though some cases may get cumbersome, as well as the assertion about the signs. That approach would give that intersection pairs of types (2) and (3) of $\alpha$ are also present in $\gamma$, maintaining their types. Since $\Gamma$ has valence $3$ at every branching point, $\gamma$ has no intersection pairs of type (1), so the type (1) linked pairs of $\alpha$ will become linked pairs of type either (2) or (3) in $\gamma$. Some of the unlinked intersection pairs of type (1) of $\alpha$ may indeed be removed when passing to $\gamma$ (e.g. a suitable parametrization of a circle with three points identified). We do not need these assertions to show Theorem \ref{simple}, so we will not prove them.

\subsection{Signs and conjugation}

Next we show the main lemma that implies there are no cancellations in the formula for $[X,X^{-1}]$ when $X$ is primitive. After that, we shall finish the details of the proof of Theorem \ref{simple}. Let $\alpha\in\Omega$ be a cyclically reduced non-trivial closed curve, and $x=\alpha(0)$. We consider $C=\alpha(I)$ and $C_{\varepsilon}$ as in the rest of this section, identifying $\mathcal E_x(C)$ with $\tilde C$ embedded in $\tilde C_{\varepsilon}$.

\begin{lemma}\label{general2} Assume that $\alpha$ is primitive, and that $g, h\in T_0(\alpha)$ are such that $[\alpha]h[\alpha]^{-1}h^{-1}$ and $[\alpha]g[\alpha]^{-1}g^{-1}$ are conjugate. Then 
\[\epsilon_g=\epsilon_h. \]
\end{lemma}

\begin{proof}

If $g$ and $h$ reverse orientation this is a consequence of Lemma \ref{reversing2}. So we assume that $g$ and $h$ preserve orientation, since by Remark \ref{r:int length} this is the other possible case. Recall the notation of Equations \eqref{c1h} and \eqref{c1g}, and let $\phi$ be the element defined in equation \eqref{8}. By Lemma \ref{general1} we have \begin{equation} \label{c1c2 - final}
c_i(\alpha,g)=c_i(\alpha,h) \qquad \mbox{for } i=1,2, \end{equation}
so $r$ and $t$ are constant in Equations \eqref{c1h} and \eqref{c1g}, and we have $\phi=[\gamma_h\gamma_g^{-1}]$. Therefore we get that $\phi b_g(1)=b_h(1)$.

Let $\tilde c_1(\alpha, g)$ be the lift of $c_1(\alpha, g)$ starting at $b_g(1)$, and $\tilde c_2(\alpha, g)$ the lift of $\tilde c_2(\alpha, g)$ ending at $b_g(1)$. Then $\tilde c_1(\alpha, g)$ is contained in $\Lambda_{\alpha}$ with positive orientation, and $\tilde c_2(\alpha, g)$ only meets $\Lambda_{\alpha}$ at its endpoint $b_g(1)$, since $g$ preserves orientation. By Remark \ref{rm:signs2}, the sign $\epsilon_g$ is decided by the side of $\Lambda_{\alpha}$ that $\tilde c_2(\alpha, g)$ is on. We may write 
\begin{equation} \label{e-g}
 \epsilon_g = sign(b_g\tilde c_1(\alpha, g), \tilde c_2(\alpha, g)) \end{equation} where we understand we are taking the intersections of these curves with a small enough ball centered at $b_g(1)$, as to follow the definition of sign in Subsection \ref{LPs}. 

Similarly, we take $\tilde c_1(\alpha, h)$ as the lift of $c_1(\alpha, h)$ starting at $b_h(1)$, and $\tilde c_2(\alpha, h)$ as the lift of $\tilde c_2(\alpha, h)$ ending at $b_h(1)$. The same observations hold, in particular \begin{equation} \label{e-h}
 \epsilon_h = sign(b_h\tilde c_1(\alpha, h), \tilde c_2(\alpha, h)) \end{equation}

Since $\phi b_g(1)=b_h(1)$ and deck transformations preserve horizontal lifting, Equation \eqref{c1c2 - final} implies that 
\begin{equation} \label{tripod}
 \phi \tilde c_1(\alpha, g)= \tilde c_1(\alpha, h)\;\;\;\;\; \mbox{and} \;\;\;\;\;  \phi \tilde c_2(\alpha, g)= \tilde c_2(\alpha, h).    \end{equation}
The situation is depicted in Figure \ref{fig:5}. 

\begin{figure}[htbp]
\input{fig5.pic}
\caption{Proof of lemma \ref{general2}}\label{fig:5}
\end{figure}

First we shall complete the proof assuming the following extra condition:

{\sl Assumption 1:} There exists $\eta$ a final segment of $b_g$ so that $\phi\eta$ is a final segment of $b_h$.

Under this assumption, we note that Equation \eqref{tripod} gives that $$ sign(\eta\tilde c_1(\alpha, g), \tilde c_2(\alpha, g))= sign(\phi\eta\cdot\tilde c_1(\alpha, h), \tilde c_2(\alpha, h))  $$ since $\phi$, as a deck transformation, preserves the orientation of $\tilde C_{\varepsilon}$. When we combine this with Equations \eqref{e-g} and \eqref{e-h}, Assumption 1 gives us $$\epsilon_g=\epsilon_h $$ as desired.

Now observe that Assumption 1 holds if $t_g(1)=t_h(1)$ has valence $3$ in $C$: for then $b_g(1)$ and $b_h(1)$ would have valence $3$ in $\tilde C$, and since $\phi$ is bijective and satisfies Equation \ref{tripod}, we get that a small enough final segment of $b_g$ must be mapped by $\phi$ to a final segment of $b_h$. Figure \ref{fig:5} depicts this case.

Changing the basepoint of $\alpha$ if necessary, we consider the curve $\gamma$ in Lemma \ref{approximation}. By Lemma \ref{sign} we may replace $\alpha$ with $\gamma$ if needed to ensure Assumption 1. 

%see that Assumption 1 holds for the curve $\gamma$ in Lemma \ref{approximation}. we can apply Lemma \ref{sign} to replace $\alpha$ with $\gamma$ 

%In light of this, we may use Lemma \ref{sign} to replace $\alpha$ with $\gamma$ as in Lemma \ref{approximation} if necessary, noting that Assumption 1 holds for $\gamma$. (We may also need to change the basepoint of $\alpha$ previously).   

\end{proof}

We finally complete the proof of Theorem \ref{simple}. Let $X\in S(M)$ be  non-trivial and primitive, and $\alpha$ be a cyclically reduced representative of $X$.

{\sl Proof of Theorem \ref{simple}:} As we pointed out after Definition \ref{def:T0}, we have that $LP_1(X)=LP_2(X,X)\cong LP_2(X,X^{-1})$ which is in correspondence with $T_0(\alpha)$, recalling Lemma \ref{bijection} and Definition \ref{def:T0}. We rewrite the formula for $[X,X^{-1}]$ given in Definition \ref{d:bracket} using this correspondence, Lemma \ref{formula2} and Definition \ref{def:liftsign}, to get

\begin{equation} \label{corchete final}
[X,X^{-1}]=\sum_{g\in T_0(\alpha)} \epsilon_g\{[\alpha]g[\alpha]^{-1}g^{-1} \} 
\end{equation} where $\{[\alpha]g[\alpha]^{-1}g^{-1} \}$ stands for the conjugacy class of $[\alpha]g[\alpha]^{-1}g^{-1}$. 

We must show that if $[X,X^{-1}]=0$ then $LP_1(X)=\emptyset$, or equivalently, $T_0(\alpha)=\emptyset$. We show it by contradiction: we assume that $[X,X^{-1}]=0$ and $T_0(\alpha)$ is nonempty. Since $\mathcal S(M)$ is a free abelian group of basis $S(M)$, there must be cancellations in the second term of Equation \eqref{corchete final}, so there must be $g,h\in T_0(\alpha)$ with
$$\{[\alpha]g[\alpha]^{-1}g^{-1} \}=\{[\alpha]h[\alpha]^{-1}h^{-1}\} \quad \mbox{ and } \quad \epsilon_g=-\epsilon_h  $$ contradicting Lemma \ref{general2}.

\begin{flushright}
$\Box$
\end{flushright}

\vskip 1cm

\bibliography{biblio}

\bibliographystyle{amsalpha}

\end{document}